\documentclass[12pt]{amsart}
\usepackage{amsfonts,amssymb,amscd,amstext,mathrsfs}
\usepackage[utf8]{inputenc}
\usepackage{hyperref}
\usepackage{verbatim}
 
\usepackage{graphics}
\usepackage{graphicx}

\usepackage{times}
\usepackage{enumerate}
\usepackage[up,bf]{caption}
\usepackage{color}
\input xy
\xyoption{all}

\usepackage[color=blue!20]{todonotes}

\newtheorem{theorem}{Theorem}[section]
\newtheorem{proposition}[theorem]{Proposition}

\newtheorem{lemma}[theorem]{Lemma}
\newtheorem{corollary}[theorem]{Corollary}

\theoremstyle{definition}
\newtheorem{definition}[theorem]{Definition}
\newtheorem{remark}[theorem]{Remark}

\newtheorem{conjecture}[theorem]{Conjecture}
\newtheorem{problem}[theorem]{Problem}
\newtheorem{example}[theorem]{Example}

\numberwithin{equation}{section}
\numberwithin{figure}{section}

\newcommand\Acal{\mathcal{A}}

\newcommand\Hcal{\mathcal{H}}

\newcommand\Ascr{\mathscr{A}}

\newcommand\Oscr{\mathscr{O}}

\newcommand\Tscr{\mathscr{T}}

\newcommand\B{\mathbb{B}}
\newcommand\C{\mathbb{C}}
\newcommand\D{\overline{\mathbb D}}
\newcommand\CP{\mathbb{CP}}

\renewcommand\D{\mathbb D}

\newcommand\N{\mathbb{N}}

\newcommand\R{\mathbb{R}}

\newcommand\Z{\mathbb{Z}}

\newcommand\igot{\mathfrak{i}}

\renewcommand\igot{\mathfrak{i}}

\renewcommand\imath{\igot}

\newcommand\hra{\hookrightarrow}

\newcommand\wt{\widetilde}

\newcommand\di{\partial}

\newcommand\dist{\mathrm{dist}}

\newcommand\Pic{\mathrm{Pic}}

\def\dist{\mathrm{dist}}

\def\br{\mathrm{br}}

\numberwithin{equation}{section}

\makeatletter
\@namedef{subjclassname@2020}{\textup{2020} Mathematics Subject Classification}
\makeatother

\begin{document}
\title{Oka-1 manifolds}
\author{Antonio Alarc\'on \; and\; Franc Forstneri{\v c}}

\address{Antonio Alarc\'on, Departamento de Geometr\'{\i}a y Topolog\'{\i}a e Instituto de Matem\'aticas (IMAG), Universidad de Granada, Campus de Fuentenueva s/n, E--18071 Granada, Spain}
\email{alarcon@ugr.es}

\address{Franc Forstneri\v c, Faculty of Mathematics and Physics, University of Ljubljana, and Institute of Mathematics, Physics, and Mechanics, Jadranska 19, 1000 Ljubljana, Slovenia}
\email{franc.forstneric@fmf.uni-lj.si}


\date{10 April 2025}
%%%%%

%; rationally connected manifold}

\begin{abstract}
In this paper we begin a systematic study of the class of
complex manifolds which are universal targets of holomorphic maps from open 
Riemann surfaces. We call them Oka-1 manifolds, by analogy with 
Oka manifolds that are universal targets of holomorphic maps
from Stein manifolds of arbitrary dimension.
We prove that every complex manifold which is dominable at most points
by spanning tubes of complex lines in affine spaces is an Oka-1 manifold. 
In particular, a manifold dominable by $\C^n$ at most points is an Oka-1 manifold. We provide many examples of Oka-1 manifolds among compact 
complex surfaces, including all Kummer surfaces and all elliptic K3 surfaces.
We show that the class of Oka-1 manifolds is invariant under Oka-1 maps inducing 
a surjective homomorphism of fundamental groups; this includes 
holomorphic fibre bundles with connected Oka fibres.
In another direction, we prove that every bordered Riemann surface admits a 
holomorphic map with dense image in any connected complex manifold. 
The analogous result is shown for holomorphic Legendrian immersions in an arbitrary
connected complex contact manifold.
\end{abstract}

\maketitle

\setcounter{tocdepth}{1}
\tableofcontents

%
%
%    INTRODUCTION
%
%
\section{Introduction}\label{sec:intro} 

%
% FF   I have slightly changed and expanded the introductory page.
%
The study of holomorphic curves is a perennial subject in complex 
and algebraic geometry. In this paper, we introduce and investigate a 
new class of complex manifolds characterized
by the property that they admit plenty of holomorphic curves parameterized 
by any open Riemann surface. 
Here is the precise definition of this class of manifolds.

%
%  OKA-1
%
\begin{definition}\label{def:Oka1} 
A connected complex manifold $X$ 
is an {\em Oka-1 manifold} if for any open Riemann surface $R$, 
compact Runge set $K$ in $R$, discrete sequence $a_i \in R$ without repetitions, 
continuous map $f:R\to X$ which is holomorphic on a neighbourhood of 
$K\cup \, \bigcup_i\{a_i\}$, number $\epsilon>0$, and integers 
$k_i\in\N=\{1,2,\ldots\}$ there is a holomorphic map 
$F:R\to X$ which is homotopic to $f$ and satisfies
\begin{enumerate}
\item $\sup_{p\in K} \dist_X(F(p),f(p)) < \epsilon$, and 
\item $F$ agrees with $f$ to order $k_i$ in the  point $a_i$ for every $i$.
\end{enumerate}
If condition (1) can be satisfied then $X$ has the 
Oka-1 property with approximation. 
%If in addition condition (2) holds with $k_i=1$ then $X$ has  the Oka-1 property with approximation and interpolation. 
A complex manifold $X$ is Oka-1 if every component of $X$ is such.
\end{definition}

Here, $\dist_X$ is a distance function on $X$ inducing its manifold topology. 
Since the approximation condition (1) takes place on a compact
subset, being Oka-1 is independent of the choice of $\dist_X$. 
Recall that a compact set $K$ in a connected 
Riemann surface $R$ is said to be Runge if the complement $R\setminus K$ 
has no relatively compact connected components. 
 
The definition of an Oka-1 manifold 
is inspired by the notion of an Oka manifold, 
which developed naturally from the Oka--Grauert--Gromov theory. 
A complex manifold $X$ is an Oka manifold 
if it satisfies the natural approximation and interpolation conditions
for maps $S\to X$ from Stein manifolds $S$ 
of arbitrary dimension. (See \cite[Chapter 5]{Forstneric2017E} and 
\cite{Forstneric2023Indag,ForstnericLarusson2011}.) 
Thus, every Oka manifold is also an Oka-1 manifold but 
the converse fails. For example, there is a discrete
set $A\subset\C^2$ whose complement is not dominable by $\C^2$
(see \cite{RosayRudin1988} or \cite[Sect.\ 4.7]{Forstneric2017E}), hence 
it fails to be Oka, but it is Oka-1 by a general position argument 
(see Corollary \ref{cor:thin}).
The classical results for holomorphic functions 
on open Riemann surfaces due to Runge \cite{Runge1885}, 
Weierstrass \cite{Weierstrass1876}, 
Behnke and Stein \cite{BehnkeStein1949},  
and Florack \cite{Florack1948} show that $\C^n$ is an Oka-1 manifold
for every $n\in \N$. The Cartesian product of Oka-1 manifolds
is obviously Oka-1.

In this paper, we investigate the class of Oka-1 manifolds by 
combining techniques from complex analysis and complex and algebraic geometry.

Here are some immediate observations. % about this class.
If $X$ is an Oka-1 manifold then for every point $x\in X$ 
and tangent vector $v\in T_x X$ there exists an entire map $f:\C\to X$ with $f(0)=x$ 
and $f'(0)=v$. Hence, the Kobayashi pseudometric of $X$ vanishes 
identically and every bounded plurisubharmonic function on $X$ is constant,
i.e., $X$ is Liouville. It is even strongly Liouville; see Corollary \ref{cor:SL}.
Assuming that $X$ is connected, it admits holomorphic maps with % everywhere 
dense images from any open Riemann surface, in particular, from $\C$.
This implies that the class of compact K\"ahler or projective Oka-1 manifolds 
is conjecturally related to several important classes of complex manifolds studied 
in the literature, such as the special manifolds in the sense of Campana
\cite{Campana2004AIF,Campana2004AIF-2}; see the discussion 
by Campana and Winkelmann in \cite{CampanaWinkelmann2023}. 

The conditions in Definition \ref{def:Oka1} easily imply that a homotopy from 
the initial map $f$ to a holomorphic map $F$ can be chosen to consist of 
maps $R\to X$ which are holomorphic 
on a neighbourhood of $K\cup \bigcup_i\{a_i\}$
and agree with $f$ to order $k_i$ at $a_i$ for every $i$.
Note however that the axiom of Oka-1 manifolds does not include  
the parametric case concerning families of maps $R\to X$.
While the parametric Oka property follows from the basic Oka
property \cite[Proposition 5.15.1]{Forstneric2017E}, the proof 
uses the basic Oka principle for maps from Stein manifolds of arbitrary dimension,
and it does not apply to Oka-1 manifolds. One could introduce 
and study the class of complex manifolds satisfying the parametric 
Oka-1 property, but most of the techniques developed in this paper 
do not apply to this case.

%
% RESULTS
%
%
We now describe our main results, deferring the precise statements to individual sections.  

In Section \ref{sec:tubes} we introduce a geometric sufficient condition on a complex 
manifold to be Oka-1; see Theorem \ref{th:main1}. 
This condition concerns dominability of the manifold
by spanning tubes of complex lines in affine spaces $\C^n$. 
It holds on any complex manifold $X$ which is dominable by $\C^n$ 
at every point in a Zariski open domain, so every such
manifold is Oka-1; see Corollary \ref{cor:dominable}. 
In particular, a connected algebraic manifold which is algebraically dominable 
by $\C^n$ is Oka-1. Dominability by tubes of lines is a considerably 
weaker condition than any of the sufficient conditions in the theory 
of Oka manifolds, and is the first known condition implying an 
Oka-type property which is local in the Hausdorff topology. 
(The Oka property is Zariski local according 
to Kusakabe \cite[Theorem 1.4]{Kusakabe2021IUMJ}.) 

After the preparatory Sections \ref{sec:trees}--\ref{sec:noncritical},
Theorem \ref{th:main1} is proved in Section \ref{sec:proof} as a special case 
of Theorem \ref{th:main2}. The proof reveals several features of independent interest.
In particular, Proposition \ref{prop:weakOka1} shows that, to establish 
the Oka-1 property of a complex manifold $X$, it suffices to show that every 
holomorphic map $K\to X$ from a neighbourhood of a compact set $K$
with piecewise smooth boundary in an open Riemann surface $R$ can be approximated
by holomorphic maps $L=K\cup D\to X$, where $D\subset R$ is any compact disc attached 
to $K$ along a boundary arc. This resembles the convex approximation property,
CAP, characterizing the class of Oka manifolds (see \cite[Section 5]{Forstneric2017E}).
Here, we only need to approximate maps from one-dimensional domains which,
however, may be topologically nontrivial. 

In Section \ref{sec:functorial} we study functorial properties of the class 
of Oka-1 manifolds. %, and in particular its behaviour under holomorphic maps. 
Among the main results of the section are Theorem \ref{th:updown}, which shows that 
the class of Oka-1 manifolds is invariant under Oka maps inducing a surjective 
homomorphism of fundamental groups, and Corollary \ref{cor:updown}
which gives the same conclusion for Oka-1 maps; see Definition \ref{def:Oka1map}.
In particular, if $h:X\to Y$ is a holomorphic 
fibre bundle with a connected Oka fibre, then $X$ is an Oka-1 manifold 
if and only if $Y$ is an Oka-1 manifold. Recall that Oka maps preserve the
class of Oka manifolds; see \cite[Theorem 3.15]{Forstneric2023Indag}. 
We show by examples that Oka-1 manifolds are in general not open or 
closed in smooth families of manifolds. We also introduce the class
LSAP of complex manifolds having the {\em local spray approximation
property}; see Definition \ref{def:LSAP}. This condition 
is Hausdorff local, it holds on every Oka manifold and is 
invariant under Oka maps, it implies the conclusion of
Proposition \ref{prop:weakOka1} and hence the Oka-1 property, and it seems
to have nontrivial functorial properties that remain to be fully explored.

The question of holomorphic dominability of complex surfaces by $\C^2$
was studied in the seminal paper \cite{BuzzardLu2000} by Buzzard and Lu.
In Section \ref{sec:surfaces} we combine their results, 
and some extensions obtained by inspection of their proofs, 
with the analytic methods developed in this paper to summarize 
what we know about which such surfaces are Oka-1. 
We show in particular that the class of Oka-1 manifolds 
contains most compact complex surfaces of Kodaira dimension $-\infty$, 
all Kummer surfaces and elliptic K3 surfaces, and many elliptic surfaces of 
Kodaira dimension $1$. It turns out that for some classes of 
compact complex surfaces,  
the conditions of being Oka, Oka-1, dominable by $\C^2$, 
and having a Zariski dense entire line $\C\to X$, are equivalent. 
We expect that this is a low dimensional phenomenon and 
that the gaps between these conditions increase with the 
dimension. 

In Section \ref{sec:RC} we discuss the conjecture that every rationally 
connected projective manifold is an Oka-1 manifold,
and we mention partial results in this direction. 

%
%  Franc: I commented the following text since it need not be repeated.
%
\begin{comment}
Evidence that this 
may hold true is given by the results of Campana and Winkelmann
\cite{CampanaWinkelmann2023}, who 
constructed holomorphic lines $\C\to X$ with given jets 
through any given sequence of points in a rationally connected manifold,
and by the recent result of Benoist and Wittenberg 
\cite[Theorem 1.2]{BenoistWittenberg2025} which shows that
Conjecture \ref{con:RC} holds true for projective rationally 
simply connected manifolds. 
As explained in Sect.\ \ref{sec:RC}, an affirmative answer to 
Conjecture \ref{con:RC} follows from Proposition \ref{prop:weakOka1} and 
a theorem of Gournay \cite[Theorem 1.1.1]{Gournay2012}; however,
we could not understand the details of his proof.
\end{comment}
%
Finally, in Section \ref{sec:dense} we prove that for every connected complex
manifold $X$ and open bordered Riemann surface $M$ there exist
holomorphic curves $M \to X$ passing through any given sequence
of points in $X$. Essentially the same proof, together with the main result
of \cite{Forstneric2022APDE}, gives the analogous statement for holomorphic
Legendrian immersions from bordered Riemann surfaces to any connected 
complex contact manifold. The existence of proper holomorphic
maps, immersions, and embeddings from bordered Riemann surfaces 
to certain noncompact complex manifolds was studied in \cite{DrinovecForstneric2007DMJ,Forstneric2023Roumaine}.

The results in this paper also hold for holomorphic maps from open 
1-dimensional complex spaces, since every such space is normalized 
by an open Riemann surface.

%
% Franc: A new paragraph added on April 9, 2025
%
Since the initial version of this paper was posted on the arXiv in 2023,
several works appeared citing the preprint.
In \cite{ForstnericLarusson2025MZ}, Forstneri\v c and L\'arusson 
discovered new examples and functorial properties of Oka-1 manifolds. 
They also formulated and studied the algebraic version of the Oka-1 condition, 
called aOka-1, which concerns approximation and 
interpolation of holomorphic maps from open subsets 
of affine algebraic curves to algebraic manifolds by regular algebraic maps.
They showed that aOka-1 property is a birational invariant for compact 
algebraic manifolds and holds for all rational manifolds
and all algebraically elliptic projective manifolds 
\cite[Theorem 1.6]{ForstnericLarusson2025MZ}. 
Every aOka-1 manifold is also an Oka-1 manifold as follows
from Proposition \ref{prop:weakOka1} (b);
see \cite[Proposition 1.9]{ForstnericLarusson2025MZ}.
Every projective aOka-1 manifold is 
rationally connected \cite[Proposition 1.7]{ForstnericLarusson2025MZ}.
Very recently, Benoist and Wittenberg
\cite[Theorem 1.2]{BenoistWittenberg2025} showed that every 
projective rationally simply connected manifold is an aOka-1 manifold. 
(See the more precise discussion following Conjecture \ref{con:RC}.)
They showed that this class contains every smooth
hypersurface of degree $d$ in $\CP^n$ with ${n\geq d^2-1}$
\cite[Corollary 1.3]{BenoistWittenberg2025}. 
Oka-1 manifolds were also used by Guo and Xie 
\cite{GuoXie2024MA} in their construction of universal holomorphic maps 
with slow growth, and by Xie \cite{Xie2024IMRN} 
in the construction of entire curves producing distinct Nevanlinna currents.

%
%
%  SECTION: DOMINABILITY BY TUBES OF LINES
%
%
\section{A complex manifold densely dominable by tubes of lines is Oka-1}\label{sec:tubes}

In this section, we introduce a geometric condition implying the Oka-1 property. 
It is based on the notions of a tree and a tube of complex lines, 
and of dominability by such tubes. It holds in particular on any 
complex manifold $X$ which is dominable by $\C^{\dim X}$ at most points. 

An affine complex line in $\C^n$ is a set of the form 
$\Lambda=\{a+tv:t\in\C\}=a+\C v$, where $a\in \C^n$ and 
$v\in \C^n\setminus \{0\}$ is a {\em direction vector} of $\Lambda$.

%
%  DEFINITION OF A TREE AND A TUBE
%
\begin{definition}\label{def:tree}
A {\em tree of lines} in $\C^n$ is a connected set $\Lambda=\bigcup_{i=1}^k \Lambda_i$ 
whose {\em branches} $\Lambda_i$ are affine complex lines with linearly independent 
direction vectors $v_i\in\C^n$. 
The tree $\Lambda$ is {\em spanning} if $k=n$; equivalently, if 
the direction vectors $v_1,\ldots,v_k$ are a basis of $\C^n$. 
An open connected neighbourhood $T\subset \C^n$ of a tree of lines $\Lambda$ 
is a {\em tube of lines} around $\Lambda$ if $T$ is a union of affine translates 
of $\Lambda$. Such a tube $T$ is {\em spanning} if the tree $\Lambda$ is spanning.
\end{definition}

%
% FF   Since dominability is really a standard notion, no definition is required.
%        We only point out that more general manifolds will be used.
%
A complex manifold $X$ is said to be {\em dominable} by 
a complex manifold $Z$ at a point $x\in X$ if there exists a holomorphic map
$F:Z\to X$ and a point $z\in Z$ such that $F(z)=x$ and the differential
$dF_z:T_z Z\to T_xX$ is surjective; this requires that $\dim Z\ge \dim X$. 
The classical notion of dominability refers to the case $Z=\C^n$ with $n=\dim X$.

\begin{comment}
The following is the geometric condition we shall work with.
%
% Definition
%
\begin{definition}\label{def:dominable-tubes}
A complex manifold $X$ of dimension $n$ is said to be {\em dominable  by tubes of lines (at a point $x\in X$, densely, or strongly)} if $X$ is dominable (at the point $x\in X$, densely, or strongly) by the collection $\Acal^n$ consisting of all spanning tubes of lines in all complex Euclidean spaces of dimension at least $n$.
\end{definition}

We have that $\Acal^n=\bigcup_{N\ge n}\Acal_N$, where $\Acal_N$ denotes the collection of all spanning tubes of lines $T\subset\C^N$. 
Each type of dominability by tubes of lines is implied by the corresponding dominability 
by $\C^n$ with $n=\dim X$.
\end{comment}

We denote by $\Hcal^{k}=\Hcal^{k}_{X,g}$ the $k$-dimensional Hausdorff measure on
a manifold $X$ with respect to a Riemannian metric $g$ on $X$; see 
Federer \cite{Federer1969} or Morgan \cite{Morgan2009} for this notion. 

Here is our first main result. Note that the choice of the Riemannian metric $g$ on $X$ is 
irrelevant in the condition $\Hcal^{2n-1}(E)=0$ used in the theorem. 

%
%  The main theorem 
%
%\begin{theorem}\label{th:main1}
%A complex manifold densely dominable by tubes of lines is an Oka-1 manifold. 
%\end{theorem}

\begin{theorem}\label{th:main1}
Let $X$ be a complex manifold of dimension $n$. Assume that 
there is a closed subset $E\subset X$ with $\Hcal^{2n-1}(E)=0$ such
that at every point $x\in X\setminus E$, $X$ is dominable by a spanning tube 
of lines $T\subset\C^N$ (possibly depending on $x$).
Then, $X$ is an Oka-1 manifold. 
\end{theorem}

\begin{remark}\label{rem:dominable}
When the condition in the theorem holds, we say that 
$X$ is {\em densely dominable} by (spanning) tubes of lines. If this holds with $E=\varnothing$, we say that $X$ is {\em strongly dominable} by tubes of lines. 
Note that $\C^n$ itself is a spanning tube of lines.  
\end{remark}

We wish to point out that dense (and even strong) dominability by tubes of lines, 
or by Euclidean spaces, is a considerably weaker condition  
than any of the sufficient conditions used in the theory of Oka manifolds.  
It is the first known condition implying an Oka-type property of a complex
manifold $X$ which is local in the Hausdorff topology, in the sense that dominability 
is required at a single point of $X$ at a time. For domains in $\C^n$, 
a similar condition was considered in \cite[Theorem 1.12]{DrinovecForstneric2023JMAA},
where the main focus was on the study of domains in $\R^n$ 
satisfying a similar property with respect to minimal surfaces.

Theorem \ref{th:main1} is proved in Section \ref{sec:proof} in a more general
form; see Theorem \ref{th:main2}. Sections \ref{sec:trees}--\ref{sec:noncritical}
contain the preparatory technical lemmas. In the remainder of this section, 
we discuss applications of this theorem and its relationship to extant results in the literature.

The following obvious corollary was the vantage point of our investigations.

\begin{corollary}\label{cor:Cn}
If $\Omega$ is a connected domain in $\C^n$ such that every point 
$z\in \Omega$ is contained in a spanning tube of lines $T_z \subset \Omega$,
then $\Omega$ is an Oka-1 manifold. In particular, every spanning tube
of lines in $\C^n$ is an Oka-1 manifold.
\end{corollary}

A domain $\Omega$ in a complex manifold $X$ is said to be {\em Zariski open}
if the complement $X\setminus \Omega$ is a proper closed complex subvariety of $X$.
This is the holomorphic analogue of the standard notion in the 
category of complex algebraic manifolds. 
Let us record several further observations which follow from Theorem \ref{th:main1}.

%
%  MAIN THEOREM: DOMINABLE MANIFOLDS ARE OKA-1
%
\begin{corollary}\label{cor:dominable}
Let $X$ be a connected complex manifold of dimension $n$.
\begin{enumerate}[\rm (a)]
\item 
If $X$ is densely dominable by $\C^n$, then $X$ is an Oka-1 manifold. 
\item 
If $X$ is dominable by $\C^n$ at every point in a Zariski open subset, 
or if it contains a Zariski open Oka domain, then $X$ is an Oka-1 manifold.
\item 
If $Y$ is a complex manifold of dimension $n=\dim X$ 
which is densely dominable by tubes of lines and
$h:Y\to \Omega$ is a proper holomorphic map onto a Zariski open 
subset $\Omega$ of $X$, then $X$ is an Oka-1 manifold.
\item If $h:Y\to X$ is a surjective holomorphic map of compact complex manifolds 
of the same dimension and $Y$ is densely dominable, % (by tubes of lines, or by $\C^n$), 
then $X$ is an Oka-1 manifold.
\item 
A connected algebraic manifold algebraically dominable by $\C^n$ is an Oka-1 manifold.
\end{enumerate}
\end{corollary}
%
% Proof
%
\begin{proof}
Part (a) follows directly from Theorem \ref{th:main1}.  
To obtain (b), note that an Oka domain $\Omega\subset X$ is dominable by $\C^n$ with 
$n=\dim X$ at every point. If $\Omega$ is Zariski open in $X$ then 
$A=X\setminus \Omega$ is a closed complex subvariety of $X$ with $\Hcal^{2n-1}(A)=0$.
Hence, $X$ is densely dominable by $\C^n$, so it is Oka-1 by (a). 
To prove (c), assume that $E$ is a closed subset 
of $Y$ with $\Hcal^{2n-1}(E)=0$. Then, $E'=h(E\cup \br\, h)$ 
is a closed subset of $\Omega$ with $\Hcal^{2n-1}(E')=0$. 
(Here, $\br\, h$ denotes the branch locus of a holomorphic map $h$.)
Since $A=X\setminus \Omega$ is a proper complex subvariety
of $X$, $A\cup E'$ is a closed subset of $X$ with 
$\Hcal^{2n-1}(A\cup E')=0$. Note that $X$ is dominable at every 
point of $X\setminus A\cup E'$, so the conclusion follows from (a).
Part (d)  is an obvious consequence of (c). Finally, to see (e) note that 
if $X$ is an algebraic manifold and $F:\C^n\to X$ is a dominating algebraic map
then $F(\C^n\setminus \br\, F)$ is a Zariski open domain in $X$. Hence, $X$ is
densely dominable by $\C^n$, and thus an Oka-1 manifold by (a).
\end{proof}

Examples show that parts (c) and (d) in the corollary 
fail in general if $\dim Y>\dim X$.
However, given an Oka map $h: Y\to X$ which induces
a surjective homomorphism of fundamental groups, $X$ is Oka-1 
if and only if $Y$ is Oka-1; see Theorem \ref{th:updown}. 
A connected Oka manifold $X$ is strongly dominable 
by a holomorphic map $F:\C^n\to X$ with $n=\dim X$, in the sense that 
$F(\C^n\setminus \br\, F)=X$ (see \cite{Forstneric2017Indam}). 
It is not known whether every complex $n$-manifold which is strongly dominable 
by $\C^n$ is an Oka manifold; see \cite[Section 7.1]{Forstneric2017E}.
% and in particular the question at the end of the section in p.\ 325. 

A comprehensive study of complex surfaces holomorphically dominable by 
$\mathbb C^2$ was made by Buzzard and Lu \cite{BuzzardLu2000}. 
%Although we are unable to draw the conclusion of Corollary \ref{cor:dominable} (e) underholomorphic dominability (indeed, dominating holomorphic maps can have rather small range as shown by Fatou--Bieberbach domains), 
%
Inspection shows that the complex surfaces for which dominability 
is established in their paper are actually densely dominable by $\C^n$, hence Oka-1.
%so Corollary \ref{cor:dominable} (a) applies. 
We discuss these applications in Section \ref{sec:surfaces}. 
Among the highlights, we mention that every Kummer surface and every
elliptic K3 surface is an Oka manifold; see Proposition \ref{prop:Kummer} and 
Corollary \ref{cor:K3elliptic}.

When $\dim X=1$, i.e., $X$ is a Riemann surface, dominability by 
tubes of lines is clearly equivalent to dominability by $\C$, 
which holds if and only if $X$ is one of the surfaces 
$\CP^1,\C,\C^*=\C\setminus\{0\}$, or a torus.
These are precisely the Riemann surfaces which are Oka manifolds;
see \cite[Corollary 5.6.4]{Forstneric2017E}.  
Summarizing, we have the following observation.

%
%  RIEMANN SURFACES
%
\begin{corollary}\label{cor:RS}
For a Riemann surface $X$ the following properties are pairwise equivalent.
\begin{itemize}
\item $X$ is an Oka manifold.
\item $X$ is an Oka-1 manifold.
%
% AA: I added the following condition to the list; I am not sure whether it's worth
%
%\item $X$ is strongly dominable by $\C$.
\item $X$ is dominable by $\C$.
\item $X$ is not Kobayashi hyperbolic.
\item $X$ is one of the Riemann surfaces $\CP^1,\C,\C^*$, or a torus.
\end{itemize}
\end{corollary}

The proof of Theorem \ref{th:main1}, given in Section \ref{sec:proof}, 
shows that the following ostensibly weaker approximation and interpolation 
conditions imply Oka-1 properties (see Remark \ref{rem:weakOka1}).
We shall say that a map $f:K\to X$ is holomorphic on a compact set 
$K$ if it is the restriction to $K$ of a holomorphic map on a neighbourhood 
of $K$ in the ambient manifold.

%
%   DISC-APPROXIMATION PROPERTY
%
\begin{proposition}\label{prop:weakOka1}
Let $X$ be a connected complex manifold.
\begin{enumerate}[\rm (a)]
\item Assume that for any open Riemann surface $R$ and pair of 
compact sets $K\subset L$ in $R$ with piecewise smooth boundaries such that 
$D=L\setminus \mathring K$ is a disc attached to $K$ along an arc 
$\alpha\subsetneq bD$, every holomorphic map $f:K\to X$ can be  
approximated uniformly on $K$ by holomorphic maps $\tilde f:L\to X$.
Then, $X$ has the Oka-1 property with approximation.
\item If in addition the map $\tilde f$ in part (a) can be chosen such that it agrees 
with $f$ in a given finite set of points in $K$, then $X$ has the Oka-1 property 
with approximation and interpolation. If in addition jet interpolation is 
possible then $X$ is an Oka-1 manifold.
\end{enumerate}
\end{proposition}

The conditions in the proposition obviously imply
the analogous conditions if $L\setminus \mathring K$ is a union of annuli.
Indeed, attaching an annulus along a boundary component of $K$ amounts to 
successively attaching a pair of discs. This is the noncritical case in the proof of 
Theorem \ref{th:main1} (see Case 1), which holds by Lemma \ref{lem:noncritical} 
if $X$ is densely dominable by tubes of lines.

%In Section \ref{sec:functorial} we use Proposition \ref{prop:weakOka1} together with results from Oka theory to study invariance of Oka-1 manifolds under holomorphic maps.

%
%   REMARK ON GENERAL POSITION
%
\begin{remark}\label{rem:transversality}
If $X$ is a complex $n$-dimensional manifold and $E$ is a closed subset 
of $X$ with $\Hcal^{2n-2}(E)=0$, then a generic holomorphic map $f:M\to X$ 
from a compact bordered Riemann surface avoids $E$. 
Indeed, take a holomorphic submersion $F: M\times \B^N\to X$ 
%for some $N\ge n$ 
as in \eqref{eq:F}, with $F(\cdotp,0)=f$, where $\B^N$ denotes the unit ball 
in $\C^N$. Since $\dim_\R  M\times \B^N=2N+2$,
the condition $\Hcal^{2n-2}(E)=0$ implies $\Hcal^{2N} (F^{-1}(E))=0$ by Fubini's theorem, 
and hence the set of parameters $t\in\B^N$ for which the range of the
map $f_t=F(\cdotp,t):M\to X$ intersects $E$ has zero $2N$-dimensional measure. 
%
% FF   I added the following, which is needed in the proof.
%
Likewise, if $\Hcal^{2n-1}(E)=0$, the same argument shows that 
a generic holomorphic map $f:M\to X$ satisfies $f(bM)\cap E=\varnothing$,
and in this case $f$ can be chosen to agree with a given holomorphic
map $M\to X$ to a given finite order at finitely many given points of 
$\mathring M=M\setminus bM$. By using this argument inductively 
in the proof of Theorem \ref{th:main1} we obtain the following corollary.

\begin{corollary}\label{cor:thin}
Let $X$ be an Oka-1 manifold of dimension $n$. If $E$ is a closed subset of $X$
with $\Hcal^{2n-2}(E)=0$, then $X\setminus E$ is an Oka-1 manifold.
This holds in particular if $E$ is a closed complex subvariety of codimension 
at least two in $X$.
\end{corollary}

The hypothesis $\Hcal^{2n-2}(E)=0$ is optimal. Indeed, 
the corollary fails in general if $E$ is a complex hypersurface. 
For example, the complement in $\CP^n$ of $2n+1$ 
hyperplanes in general position is Kobayashi hyperbolic by 
Green's theorem \cite{Green1972}, and so is the complement of 
a very general hypersurface of sufficiently high degree; see Brotbek \cite{Brotbek2017}.
There is no analogue of Corollary \ref{cor:thin} for Oka manifolds where even the question
of removability of a point is an open problem, and closed discrete sets in $\C^n$
are not removable in general. 

The jet transversality theorem shows that a generic holomorphic map 
$M\to X$ from a compact bordered Riemann surface to an arbitrary complex
manifold $X$ is an immersion if $\dim X>1$ and an injective immersion if $\dim X>2$; 
see \cite[Section 8.8]{Forstneric2017E}. Using this fact inductively 
in the proof of Theorem \ref{th:main1} gives the following corollary. 

%
% DENSE IMMERSIONS
%
\begin{corollary}\label{cor:densecurves}
If $X$ is an Oka-1 manifold of dimension $n>1$ then holomorphic maps 
$F:R\to X$ in Definition \ref{def:Oka1} can be chosen to be immersions 
(injective immersions if $n>2$) if the interpolation conditions allow it.
If a complex manifold $X$ of dimension $>1$ has the Oka-1 property with 
approximation, then every open Riemann surface $R$ admits a holomorphic 
immersion $R\to X$ (injective immersion if $\dim X>2$) 
with everywhere dense image. 
\end{corollary}
\end{remark}

Corollary \ref{cor:densecurves} clearly fails if the manifold $X$ is Brody hyperbolic. 
On the other hand, it was shown by Forstneri\v c and Winkelmann \cite{ForstnericWinkelmann2005MRL}
that every connected complex manifold $X$ admits a holomorphic disc $\D \to X$ 
hitting any given sequence in $X$, so there are
holomorphic discs in $X$ with everywhere dense images. In Section \ref{sec:dense} 
we generalize this to maps from any bordered Riemann surface
(see Theorem \ref{th:dense}), and also from some other open Riemann surfaces 
with more complicated topology (see Corollary \ref{co:dense-2}).

Further examples of Oka-1 manifolds are discussed in Section \ref{sec:surfaces} 
where we focus on complex surfaces. Note however that a compact complex manifold 
may be dominable by tubes of lines but not contain any rational curves. 
Examples include tori $\C^n/\Gamma$, where $\Gamma$ 
is a discrete group of translations on $\C^n$. 
These are complex homogeneous manifolds and
hence Oka manifolds (see \cite[Proposition 5.6.1]{Forstneric2017E} due to 
Grauert \cite{Grauert1957I}). By Corollary \ref{cor:Cn} and 
Proposition \ref{prop:coverings}, a torus contains many
domains which are dominable by tubes of lines. 
Another such example are Hopf manifolds.
Every Hopf manifold is an unramified quotient of $\C^n\setminus\{0\}$ $(n>1)$ by 
a cyclic group, so it is an Oka manifold \cite[Corollary 5.6.11]{Forstneric2017E}. 
Like tori, Hopf surfaces contain many Oka-1 domains. 

%
%
%   A PROBLEM
%
\begin{problem}\label{prob:dominability}
Let $T\subset \C^n$ be a spanning tube of lines for some $n\ge 2$
(see Definition \ref{def:tree}). 
\begin{enumerate}[\rm (a)] 
\item Does there exist a dominating holomorphic map $\C^n\to T$? 
\item Is $T$ an Oka manifold?
\item Is there a compact Oka-1 manifold which is not an Oka manifold?
\end{enumerate}
\end{problem}

We expect that the answers to questions (a) and (b) are negative, while the answer to (c) is 
positive in dimension $>1$. Natural candidates may be the K3 surfaces.

%
%  THE EXTENDED TUBE
%
\begin{remark}\label{rem:extendedtube}
Consider $\CP^1$ with the coordinate $z\in \C\cup\{\infty\}$. 
We claim that the extended (spanning) tube of rational curves 
\[
	\wt T=\big\{(z,w)\in \CP^1 \times \CP^1: |z|<1\ \text{or}\ |w|<1\big\}
\]
is an Oka surface, and hence by \cite[Theorem 5.5.1 (e)]{Forstneric2017E} 
there is a surjective holomorphic map $\C^2\to \wt T$.
Indeed, the complement of the tree of rational curves $\Lambda=\{z=0\}\cup\{w=0\}$ 
in $X=\CP^1 \times \CP^1$ equals $\C^2$ with the complex coordinates $(1/z,1/w)$, 
and $K=X\setminus \wt T$ is the closed bidisc $\{|1/z|\le1,\ |1/w|\le 1\}$,
which is polynomially convex in $X\setminus \Lambda\cong \C^2$. Hence, 
\[
	\wt T\setminus \Lambda = (X\setminus K)\setminus \Lambda
	= (X\setminus \Lambda)\setminus K =\C^2\setminus K
\]
is an Oka surface by Kusakabe's theorem 
\cite[Theorem 1.2 and Corollary 1.3]{Kusakabe2024AM}
(see also \cite{ForstnericWold2020MRL}). The same argument applies 
to any translation of the tree $\Lambda$ within the tube $\wt T$. Clearly, there are 
translates $\Lambda_2,\Lambda_3\subset \wt  T$ of $\Lambda=\Lambda_1$
with $\bigcap_{i=1}^3 \Lambda_i=\varnothing$. 
Hence, $\wt T= \bigcup_{i=1}^3 \wt T\setminus \Lambda_i$
is a union of Zariski open Oka domains $\wt T\setminus \Lambda_i$, 
so it is an Oka manifold by Kusakabe's localization theorem 
\cite[Theorem 1.4]{Kusakabe2021IUMJ}. 
This argument clearly fails in dimensions $n>2$.
\end{remark}

%
% ORGANIZATION
%
Let us say a few words about the proof of Theorem \ref{th:main1}.

In Section \ref{sec:trees} we collect some basic definitions and observations concerning trees 
and tubes of complex lines in affine spaces $\C^n$. 

In Section \ref{sec:bump} we obtain the first main lemma used in the proof 
(see Lemma \ref{lem:bump}), which pertains to the 
situation described in Proposition \ref{prop:weakOka1}. More precisely,
let $K$ be a compact domain with piecewise smooth boundary in an open Riemann 
surface $R$, and let $D\subset R$ be a compact disc attached to $K$
along a boundary arc $\alpha = K\cap D = bK\cap bD$ such that the set 
$L=K\cup D$ has piecewise smooth boundary. Given a spanning
tube of lines $T\subset \C^n$ and a holomorphic map $f:K\to \C^n$ 
such that $f(\alpha)\subset T$, we show that $f$ can be approximated
uniformly on $K$ by holomorphic maps $\tilde f:L\to \C^n$ such that 
$\tilde f(D)\subset T$. The analogous result holds for local holomorphic 
sprays of maps $K\to \C^n$ sending $\alpha$ to $T$;  see Remark \ref{rem:bump}. 

In Section \ref{sec:noncritical} we use Lemma \ref{lem:bump} 
and methods from Oka theory to prove Lemma \ref{lem:noncritical}, 
which provides the noncritical case in the proof of Theorem \ref{th:main1}. 
With $K\subset R$ as above, this concerns the approximation of a 
holomorphic map $f:K\to X$ by holomorphic maps $\tilde f:L\to X$, 
where $L\Supset K$ is a compact set with piecewise smooth boundary 
such that $L\setminus \mathring K$ is the union of finitely many pairwise 
disjoint compact annuli. In addition, $\tilde f$ can be chosen to agree 
with $f$ to a given finite order at a given finite set of points in $K$. 

Using Lemma \ref{lem:noncritical}, we obtain Theorem \ref{th:main1} in 
Section \ref{sec:proof} as a special case of Theorem  \ref{th:main2}.

In the paper, we shall frequently use the Mergelyan approximation theorem 
on compact sets in Riemann surfaces with interpolation at finitely many 
points, both for functions and for manifold-valued maps. 
We recall the relevant notions and terminology.
Given a compact set $K$ in a complex manifold $R$, 
a map $f:K\to X$ to another
complex manifold is said to be of class $\Ascr(K)$ if $f$ is continuous on $K$ 
and holomorphic on the interior $\mathring K$ of $K$. A map is 
said to be holomorphic on $K$ if it is holomorphic on a neighbourhood 
of $K$; this class is denoted by $\Oscr(K)$.
A compact set $K$ in $R$ is said to have the {\em Mergelyan property} 
if every function $f\in \Ascr(K)$ is a uniform limit on $K$ of functions in 
$\Oscr(K)$. If $R$ is an open Riemann surface then this holds in particular 
if $K$ is Runge in $R$ (see \cite[Corollary 7]{FornaessForstnericWold2020}). 
The following result, which we state for reader's convenience,
is \cite[Theorem 1.13.1]{AlarconForstnericLopez2021}; see also 
\cite[Sec.\ 7.2]{FornaessForstnericWold2020}.

%
%  MERGELYAN THEOREM WITH INTERPOLATION
%
\begin{theorem}\label{th:Mergelyan}
Assume that $K$ is a compact set with the Mergelyan property
in a Riemann surface $R$. Given a complex manifold $X$,
a map $f:K\to X$ of class $\Ascr(K)$, and finitely many points $a_i\in K$ 
$(i=1,\ldots,m)$, we can approximate $f$ as closely as desired uniformly
on $K$ by holomorphic maps $\tilde f:U\to X$ on open neighbourhoods
$U\subset R$ of $K$ such that $\tilde f(a_i)=f(a_i)$ for every
$i=1,\ldots,m$. In addition, $\tilde f$ can be chosen to agree with $f$ 
to any given finite order in the points $a_i\in \mathring K=K\setminus bK$.
\end{theorem}

%
%
%    SECTION: TREES OF LINES
%
%
\section{Trees and tubes of complex lines}\label{sec:trees} 

The notions of a tree and a tube of (affine complex) lines in $\C^n$ was introduced in 
Definition \ref{def:tree}. In this section we collect some observations which 
will be used in the sequel.

We can enumerate the branches $\Lambda_i$ of a tree 
$\Lambda=\bigcup_{i=1}^{k}\Lambda_i$ so that for each 
$i\ge 2$, the branch $\Lambda_i$ intersects the subtree 
$\Lambda^{i-1}=\bigcup_{j=1}^{i-1}\Lambda_j$ in a single point.
The intersections are transverse (normal crossings) 
due to linear independence of direction vectors of the branches. 
Several branches may intersect at the same point; 
we call $\Lambda$ a {\em regular tree} if this does not happen.
Note that a tree with $k$ branches is regular if and only of it has exactly 
$k-1$ singular points (simple nodes). Our definition of a tree of lines is 
similar to that of a {\em tree of rational curves} in a complex manifold $X$
(see \cite[Definition 4.23]{Debarre2001} or \cite{Kollar1995E}). However, an 
addition which is important in the proof of Theorem \ref{th:main1} is that the direction vectors 
of the branches of a tree of lines are linearly independent, and they are a
basis of $\C^n$ if and only if the tree is spanning.

By a linear change of coordinates on $\C^n$ we can map the direction vectors
$v_1,\ldots,v_k$ of a tree $\Lambda$ to the first $k$ standard unit vectors
$e_1,\ldots, e_k$, where $e_i=(0,\ldots,1,\ldots,0)$ with $1$ on the $i$-th spot.
Hence, it will suffice to consider trees of lines in coordinate directions.

%
%  EXAMPLES OF TREES
%
\begin{example}\label{ex:simpletree}
Let $z=(z_1,\ldots,z_n)$ be complex coordinates on $\C^n$.
%and let $e_1,\ldots,e_n$ be the standard basis vectors. 
For each $i=1,\ldots,n$ let
\begin{equation}\label{eq:axes}
	\Lambda_i=\{z\in\C^n: z_j=0\ \text{for all}\ j\in \{1,\ldots,n\}\setminus\{i\}\}
	= \C e_i
\end{equation}
be the coordinate axis in the $z_i$ direction.
\begin{enumerate} [\rm (a)]
\item The union $\Lambda=\bigcup_{i=1}^k \Lambda_i$ of coordinate axes 
is a tree. It is a spanning tree if and only if $k=n$, and is a regular tree if and only if $n=2$.
\item 
Let $a_1,\ldots, a_{k}\in\C$ for $1\le k<n$ be complex numbers. The set
%
%  SIMPLE TREE
%
\begin{equation}\label{eq:simpletree}
	\Lambda= \Lambda_n \cup\, \bigcup_{i=1}^{k} \big(a_i e_n+\Lambda_i\big) 
\end{equation}
is called a {\em simple tree} or a {\em comb}, and $\Lambda_n$ is the {\em stem} 
(or the {\em handle}) of $\Lambda$. It is spanning if and only if $k=n-1$, 
and is regular if and only if the numbers $a_i$ are pairwise distinct. 
Every tree of length $\le 3$ is a simple tree in suitable affine coordinates, 
but most trees of length $>3$ are not simple.
\end{enumerate}
\end{example}

%In particular, if we can order the branches of a tree so that for each $i\ge 2$ the branch $\Lambda_i$ intersects $\Lambda_{i-1}$ but does not intersects $\Lambda_j$ for $j<i-1$, then the tree is not simple if it has length more than $3$.

%We record the following observation.

%
%  NORMAL FORM FOR A TREE OF LINES
%
\begin{lemma}\label{lem:standard}
For every tree of lines in $\C^n$ there is an affine change of coordinates
which maps it to a tree of the form
\begin{equation}\label{eq:normal}
	\Lambda = \C e_n \cup \bigcup_{j=1}^l \Lambda^j,	
\end{equation}
where each $\Lambda^j$ is a tree in coordinate directions such that 
$\Lambda^j\cap \,\C e_n = a_j e_n$ for some numbers $a_1,\ldots, a_l\in\C$
(not necessarily distinct).
A tree \eqref{eq:normal} is said to be in {\em normal form}.
\end{lemma}

\begin{proof}
Pick any branch of a given tree and map it to $\C e_n$ by an affine 
transformation which maps the direction vectors of the branches to % (some of the)
coordinate vectors. Let $a_1,\ldots, a_l\in\C$ be such that $a_i e_n$ 
are the singular points of the new tree $\Lambda$. Then, $\Lambda$ satisfies the lemma.
\end{proof}

Note that an affine linear transformation of $\C^n$ maps a tree of lines $\Lambda$
to another tree of lines $\Lambda'$, and it maps a tube $T$ around $\Lambda$ 
to a tube $T'$ around $\Lambda'$. 

It will be convenient to use {\em polydisc tubes}. Let $\Delta^k\subset \C^k$ 
denote the unit polydisc. The polydisc tube of radius $r>0$ around the coordinate axis
$\Lambda_n=\C e_n$ is defined by
\[
	\Tscr_r(\Lambda_n) = \{z=(z',z_n)\in \C^{n-1}\times\C: z'\in r\Delta^{n-1}\}.
\]
For any affine complex line $\Lambda\subset \C^n$ there is an affine unitary 
change of coordinates $U:\C^n\to \C^n$ mapping $\Lambda$ onto $\Lambda_n$,
and we take $\Tscr_r(\Lambda) = U^{-1}(\Tscr_r(\Lambda_n))$. 
If $\Lambda$ is a tree of lines in the normal form \eqref{eq:normal}, 
then the polydisc tube $\Tscr_r(\Lambda)$ is defined to be the union of polydisc 
tubes of the same radius $r$ around its branches. 

%
%  BIG TREES
%
\begin{remark}
One can consider trees of lines in $\C^n$ having more than $n$ branches. However,
examples show that a spanning tree with more than $n$ branches need not contain 
a spanning tree with $n$ branches, such as those considered above. Our proof of
Theorem \ref{th:main1} does not apply if we assume dominability by spanning trees with 
more than $n$ branches. 
\end{remark}

%
%
%    EXTENDING THE MAP ACROSS A BUMP
%
%
\section{Extending a holomorphic map across a bump taking values in a tube}\label{sec:bump} 

The following lemma will be used in the proof of Theorem \ref{th:main1}. 
%The notions of a tree and a tube of complex lines were introduced in Definition \ref{def:tree}.

\begin{lemma}\label{lem:bump}
Assume that $K$ is a compact domain with piecewise smooth boundary in an open 
Riemann surface $R$, and $D$ is a compact topological disc with piecewise smooth 
boundary in $R$ such that $\alpha=D\cap K=bK\cap bD$ is an arc 
and the compact set $L=K\cup D$ has piecewise smooth boundary. 
Let $f=(f_1,\ldots,f_n):K\to \C^n$ be a map of class $\Ascr(K)$ and 
$T\subset \C^n$ be a spanning tube of lines such that $f(\alpha)\subset T$.
Then we can approximate $f$ as closely as desired uniformly on $K$ and interpolate 
it to any given finite order at a given finite set of points in $\mathring K$ by holomorphic maps 
$\tilde f:K\cup D\to\C^n$ such that $\tilde f(D)\subset T$.
\end{lemma}

\begin{proof}
Recall that $\Delta^n\subset \C^n$ denotes the unit polydisc centred at the origin.
We shall first prove the lemma under the following additional assumptions 
on $f$ and $T$: 
\begin{enumerate}[\rm (a)]
\item $f(\alpha)\subset r\Delta^n$ for some $r>0$, and  
\item $T\subset \C^n$ is the polydisc tube of radius $r$ around a tree
of lines $\Lambda$ in the normal form \eqref{eq:normal}. 
(Recall that the polydisc tube $\Tscr_r(\Lambda)$ is the union of polydisc 
tubes around its branches.)  
\end{enumerate}
These conditions on $f$ and $T$ imply that
$
	f(\alpha) \subset r\Delta^n \subset T.
$ 
Indeed, $r\Delta^n$ is contained in the tube of radius $r$ around the stem 
$\Lambda_n =\C e_n$ of the tree $T$. 

By Mergelyan theorem we may assume that $f$ is holomorphic on a 
neighbourhood of $K$ in $R$. Whenever invoking Runge or Mergelyan theorem, 
we shall also interpolate the given map in the  given finite set of
points in $\mathring K$ without mentioning it again (see Theorem \ref{th:Mergelyan}).

For simplicity of exposition, we first consider the case when $\Lambda$ is
a simple tree (a comb) of the form \eqref{eq:simpletree}. 
We begin by explaining how to choose the first $n-1$ components 
of the new map $\tilde f=(\tilde f',\tilde f_n)=(\tilde f_1,\ldots, \tilde f_n)$; 
the last component $\tilde f_n$ will be determined in the final step.
The general case when $\Lambda$ is of the form \eqref{eq:normal} will
be obtained by induction on $n$.

Let $\beta=bD\setminus \alpha$ be the complementary arc to $\alpha$ in $bD$.
Pick a closed topological disc $\Delta_0\subset D$ such that 
$\Delta_0 \cap \alpha =\varnothing$ and $\Delta_0 \cap bD$ is an arc contained 
in $\beta$. 
%(This disc will not be used in the special case of a comb, but it will come handy in the induction given in the second part of the proof.)
We extend the first component $f_1$ from $K$ to $K\cup \Delta_0$ by setting
$f_1=0$ on $\Delta_0$. By Runge theorem we can approximate $f_1$ on 
$K\cup \Delta_0$ by a holomorphic function $\tilde f_1$ on $L=K\cup D$ 
such that $|\tilde f_1|<r$ holds on $\alpha\cup \Delta_0$; see (a). 
Hence, there is a closed disc $\Delta_1\subset D$ such that 
\begin{enumerate}[\rm (i$_{1}$)]
\item $\Delta_1 \cap (\alpha\cup \Delta_0) =\varnothing$, 
\item $\overline {D\setminus \Delta_1}$ is the union of two disjoint discs 
and $b\Delta_1\cap\beta$ consists of two disjoint arcs, and 
\item $|\tilde f_1|<r$ holds on $\overline{D\setminus \Delta_1}$.
\end{enumerate}
%
% Figure
%
\begin{figure}[ht]
	\includegraphics[width=.85\textwidth]{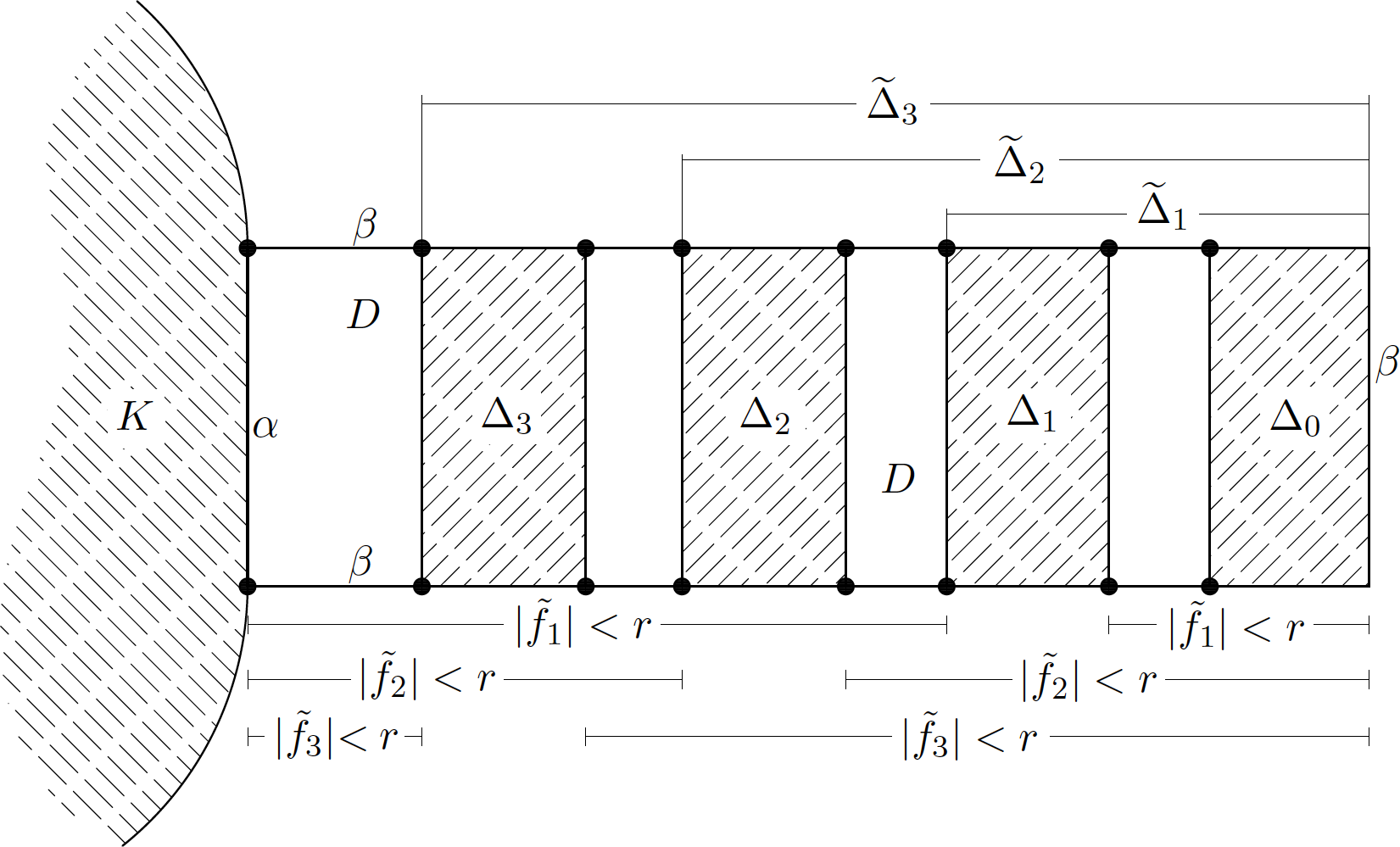}
	\caption{Proof of Lemma \ref{lem:bump} -- the special case}\label{fig:31}
\end{figure}
See Figure \ref{fig:31}.
Condition (iii$_1$) holds if the disc $\Delta_1$ satisfying conditions (i$_1$) and (ii$_1$) 
is chosen large enough. Indeed, by increasing $\Delta_1$ the set 
$\overline{D\setminus \Delta_1}$ shrinks to $\alpha\cup \Delta_0$, and we have 
that $|\tilde f_1|<r$ on $\alpha\cup \Delta_0$. 
Note that $K\cap \Delta_1=\varnothing$, and hence the 
set $K\cup \Delta_1$ is Runge in an open neighbourhood of $L=K\cup D$. 

If $n=2$, we proceed to the final argument explaining how to choose the last 
component $\tilde f_n$. Assume now that $n>2$. Let $\wt\Delta_1$ denote the union 
of $\Delta_1$ and the component of $D\setminus \Delta_1$ containing $\Delta_0$, so 
$\wt \Delta_1\subset D$ is a closed disc disjoint from $\alpha$ (see Fig.\ \ref{fig:31}).
We extend the second component $f_2$ of $f$ to 
the set $K\cup \wt \Delta_1$ by taking $f_2=0$ on $\wt \Delta_1$ and then apply 
Runge theorem on $K\cup \wt \Delta_1$ to find a holomorphic function $\tilde f_2$ 
on $L$ such that $|\tilde f_2|<r$ holds on $\alpha \cup \wt \Delta_1$; see (a). 
Hence, there is a disc $\Delta_2\subset D$ such that 
\begin{enumerate}[\rm (i$_{2}$)]
\item $\Delta_2\cap (\alpha\cup \wt \Delta_1)=\varnothing$,
\item $\overline {D\setminus \Delta_2}$ is the union of two disjoint discs 
and $b\Delta_2\cap\beta$ consists of two disjoint arcs, and 
\item $|\tilde f_2|<r$ holds on $\overline{D\setminus \Delta_2}$.
\end{enumerate}
As in the first step, condition (iii$_2$) holds if the disc $\Delta_2$ satisfying
(i$_2$) and (ii$_2$) is chosen big enough within 
$D\setminus (\alpha\cup \wt\Delta_1)$. Let $\wt\Delta_2$ denote the union of $\Delta_2$ 
and the component of $D\setminus \Delta_2$ containing $\wt \Delta_1$, so 
$\wt \Delta_2\subset D$ is a closed disc disjoint from $\alpha$. See Figure \ref{fig:31}.

If $n=3$, we proceed to the last step. If on the other hand
$n>3$, we repeat the same argument with the component $f_3$ to obtain a 
holomorphic function $\tilde f_3$ on $L$ such that 
$|\tilde f_3|<r$ holds on $\alpha\cup \wt\Delta_2$. We then pick a disc 
$\Delta_3\subset D$ which is disjoint from $\alpha\cup \wt \Delta_2$
such that $D\setminus \Delta_3$ is the union of two disjoint discs, one of them
containing $\wt\Delta_2$, and $|\tilde f_3|<r$ holds on $D\setminus \Delta_3$.
Let $\wt\Delta_3$ denote the union of $\Delta_3$ 
and the component of $D\setminus \Delta_3$ containing $\Delta_2$. See Figure \ref{fig:31}.

Clearly we can continue inductively in order to approximate the first $n-1$
components $f_1,\ldots,f_{n-1}$ of $f$ by holomorphic functions 
$\tilde f_1,\ldots,\tilde f_{n-1}$ on $L$ such that 
\begin{equation}\label{eq:cond1}
	|\tilde f_i|<r\ \ \text{holds on}\ \ \overline{D\setminus \Delta_i}
	\ \ \text{for $i=1,\ldots, n-1$}.
\end{equation}
We now extend the last component $f_n$ to the Runge compact set
$K'=K\cup\,\bigcup_{i=0}^{n-1}\Delta_i$ by setting $f_n=a_i$ on $\Delta_i$ 
for $i=0,1,\ldots, n-1$, where $a_0=0$ and the numbers $a_i\in \C$ 
for $i=1,\ldots n-1$ are as in \eqref{eq:simpletree} with $k=n-1$.
Using Runge theorem we approximate $f_n$ on $K'$
by a holomorphic function $\tilde f_n$ on $L=K\cup D$ such that 
\begin{equation}\label{eq:cond2}
	|\tilde f_n-a_i|<r\ \ \text{holds on}\ \ \Delta_i\ \ \text{for $i=0,1,\ldots,n-1$}.
\end{equation}

Conditions \eqref{eq:cond1} and \eqref{eq:cond2} imply that the map 
$\tilde f=(\tilde f',\tilde f_n) = (\tilde f_1,\ldots, \tilde f_n)$ sends the disc $D$ into the tube $T$. 
Indeed, on the disc $\Delta_i$ for $i=1,\ldots,n-1$ all components of $\tilde f'$ 
except perhaps $\tilde f_i$ are smaller than $r$ in absolute value while $|\tilde f_n-a_i|<r$, 
so $\tilde f(\Delta_i)$ is contained in the polydisc tube of radius $r$ around the affine line 
$a_n e_n+\Lambda_i\subset\Lambda$. 
On the other hand, on $D\setminus \bigcup_{i=1}^{n-1}\Delta_i$ 
all components of $\tilde f'$ are smaller than $r$, so its image by $\tilde f$ is contained 
in the polydisc tube of radius $r$ around the stem $\Lambda_n=\C e_n\subset\Lambda$. 
Note also that 
\begin{equation} \label{eq:Delta0}
	\text{$|\tilde f_j|<r$ holds on $\alpha \cup \Delta_0$ for all $j=1,\ldots,n$.}
\end{equation}

This proves the lemma (under the assumptions (a) and (b) on $f$ and $T$) 
if $\Lambda$ is a simple tree. Keeping the assumptions (a) and (b) in place, 
we now let $\Lambda$ be any tree of the form \eqref{eq:normal} with subtrees 
$\Lambda^j$ for $j=1,\ldots,l$. We proceed by induction on $n$. 
The result clearly holds for $n=2$ since in this case every tree is a comb. 
Assume inductively that $n>2$ and the result (including the condition \eqref{eq:Delta0}) 
holds in dimensions $<n$. Let $\Sigma^j$ denote the coordinate
subspace of $\C^{n-1}\times \{0\}$ spanned by the direction vectors of
the affine lines in the tree $\Lambda^j$. By renumbering the coordinates we may 
assume that $\Sigma^1$ is spanned by the coordinate directions
$1,\ldots, n_1$ for some $n_1<n$. Pick a disc $\Delta_0\subset D$ as above. 
By the inductive hypothesis, we can approximate the component functions $f_i$ for 
$i=1,\ldots,n_1$ by functions $\tilde f_i\in\Oscr(K\cup D)$ satisfying $|\tilde f_i|<r$ 
on $\alpha\cup\Delta_0$ such that the map
$\tilde f^1=(\tilde f_1,\ldots,\tilde f_{n_1}): K\cup D\to \C^{n_1}$
send $D$ into the polydisc tube $T^1\subset \C^{n_1}$ of radius $r$ around 
the tree $\Lambda^1$. Choose a disc $\Delta_1\subset D$ satisfying conditions 
(i$_{1}$)--(ii$_{1}$) above, and with condition (iii$_{1}$) replaced by
\begin{enumerate}[\rm (iii$_{1}$)]
\item $|\tilde f_i|<r$ holds on $\overline{D\setminus \Delta_1}$ for $i=1,\ldots,n_1$.	
\end{enumerate}
If $l=1$ (so $n_1=n-1$), we proceed to the last step. 
Otherwise, we repeat the same procedure for the second subtree $\Lambda^2$.
We may assume that the subspace $\Sigma^2\subset\C^n$ associated to $\Lambda^2$ 
consists of coordinate directions $n_1+1,\ldots,n_2$ for some $n_2<n$. 
Let $\wt\Delta_1$ denote the union of $\Delta_1$ and the component of 
$D\setminus \Delta_1$ containing $\Delta_0$. 
We extend the components $f_i$ of $f$ for $i=n_1+1,\ldots,n_2$ to 
the set $K\cup \wt \Delta_1$ by taking $f_i=0$ on $\wt \Delta_1$ and apply 
the inductive hypothesis to find holomorphic functions $\tilde f_i$ 
on $L$ such that $\tilde f^2=(\tilde f_{n_1+1},\ldots,\tilde f_{n_2}): L \to \C^{n_2}$
maps $D$ into the polydisc tube $T^2\subset \C^{n_2}$ of radius $r$ 
around the tree $\Lambda^2$, and $|\tilde f_i|<r$ holds on $\alpha \cup \wt \Delta_1$
for all $i=n_1+1,\ldots,n_2$. Hence, there is a disc $\Delta_2\subset D$ such that conditions 
(i$_{2}$)--(ii$_{2}$) hold, and (iii$_{2}$) is replaced by
\begin{enumerate}[\rm (iii$_{2}$)]
\item $|\tilde f_i|<r$ holds on $\overline{D\setminus \Delta_2}$ for 
$i=n_1+1,\ldots,n_2$.	
\end{enumerate}
Clearly this process continues inductively. In the last step, we choose 
$\tilde f_n$ as in \eqref{eq:cond2} with $n-1$ replaced by $l$.
We see as before that the holomorphic map 
$\tilde f=(\tilde f^1,\ldots,\tilde f^l,\tilde f_n):L \to \C^n$ sends $D$ into $T$. 
This closes the induction step and completes the proof of the lemma under the additional 
assumptions (a) and (b) made at the beginning of the proof. 

It remains to prove the general case of the lemma with the only assumption that 
$f(\alpha) \subset T$, where $T$ is a tube around an arbitrary spanning tree of lines $\Lambda$.
By Lemma \ref{lem:standard}, for every point $p\in \alpha$ there are an open neighbourhood 
$\alpha_p\subset bK$ of $p$, a tube $T_p\subset T$ containing a translate of 
$\Lambda$ through $f(p)$ such that $f(\alpha_p)\subset T_p$,
and an affine isomorphism $U_p:\C^n\to \C^n$ such that
$U_p(f(p))=0$, $U_p(T_p)$ is a tube of radius $r_p>0$ around a spanning tree
of the form \eqref{eq:normal}, and $U_p(f(\alpha_p)) \subset r_p\Delta^n$. 
In other words, the assumptions (a) and (b)
hold for the arc $\alpha_p$, the map $U_p\circ f$,
the tube $U_p(T_p)$, and the number $r_p$. 
This allows us to subdivide the arc $\alpha$ into a finite union 
$\alpha=\bigcup_{j=1}^s \alpha_j$ of closed subarcs lying back-to-back 
such that the above conditions hold on each $\alpha_j$ for an affine linear 
change of coordinates $U_j$ on $\C^n$, tube $T_j\subset U_j(T)$, and number $r_j>0$.
Let $p_{j-1}$ and $p_j$ denote the endpoints of $\alpha_j$ such that 
$p_j=\alpha_j\cap \alpha_{j+1}$ for $j=1,\ldots,s-1$. Choose an embedded arc 
$\gamma_j\subset D$ connecting the point $p_j=\alpha_j\cap \alpha_{j+1}$
to a point $q_j\in \beta=bD\setminus \alpha$ so that these arcs are pairwise disjoint,
and hence they subdivide $D$ into the union $D=\bigcup_{j=1}^s D_j$ of discs
satisfying $D_j\cap D_{j+1}=\gamma_j$ for $j=1,\ldots,s-1$ and $D_j\cap D_k=\varnothing$
if $|j-k|>1$. For notational reasons we also set $\gamma_0=p_0$ and $\gamma_s=p_s$.
(See Figure \ref{fig:31b}.)
%
% Figure 31b
%
\begin{figure}[ht]
	\includegraphics[width=.8\textwidth]{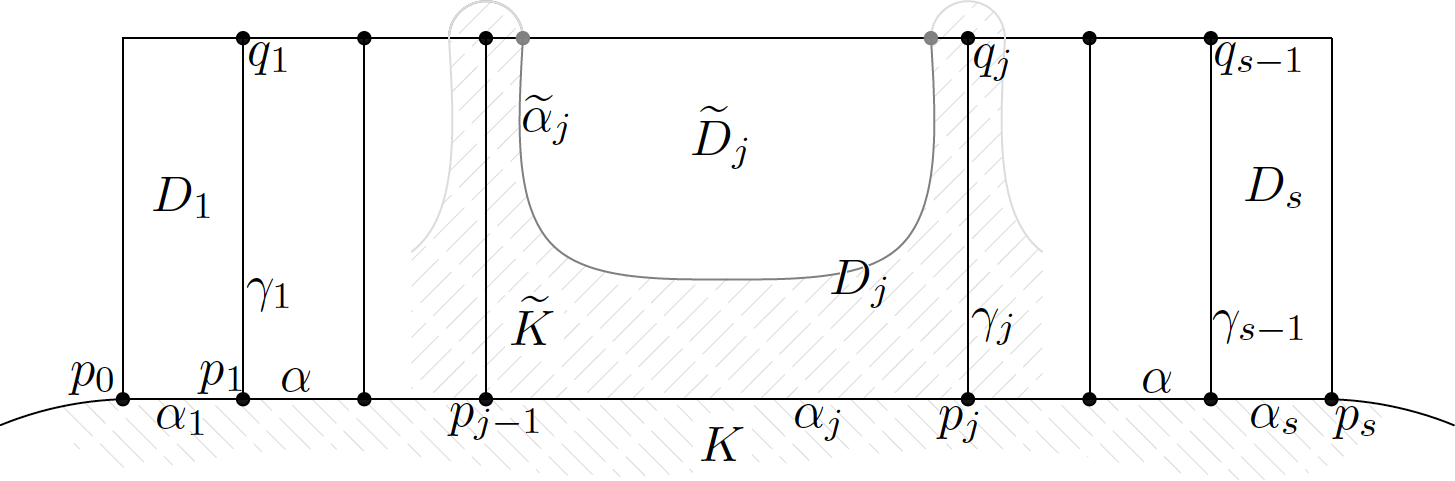}
	\caption{Proof of Lemma \ref{lem:bump} -- the general case}\label{fig:31b}
\end{figure}
We extend $f$ to the arc $\gamma_j$ as the constant map $f(p)=f(p_j)$ 
for each $p\in\gamma_j$. By Mergelyan theorem we can approximate $f$ as 
closely as desired on the compact set 
$
	S=K\cup\bigcup_{j=1}^{s-1}\gamma_j
$ 
by a holomorphic map defined on an open neighbourhood $V$ of $S$. 
Choose a small compact neighbourhood 
$\wt K\subset V$ of $S$ with smooth boundary such that 
$\overline{D\setminus \wt K}=\bigcup_{j=1}^s \wt D_j$ is the union of pairwise disjoint 
closed discs $\wt D_j \subset D_j$, and $\wt \alpha_j = b\wt D_j\cap \wt K$ is an arc
close enough to $\gamma_{j-1}\cup \alpha_j\cup \gamma_j$ such that 
\[
	U_j(f (\wt \alpha_j)) \subset r_j\Delta^n \subset T_j
	\quad \text{holds for $j=1,\ldots,s$}.
\] 
(See Figure \ref{fig:31b}.) 
It remains to apply the already established special case of the lemma
to successively extend the map (with approximation on $\wt K$
and interpolation in the  given finitely many points in $\mathring K$)
across each of the discs $\wt D_1,\ldots, \wt D_s$.  Since 
$L=K\cup D \subset \wt K\cup\, \bigcup_{j=1}^s \wt D_j$, this completes the proof.
\end{proof}

%
%   REMARK: SPRAY OF MAPS
%
\begin{remark}\label{rem:bump}
In the proof of Lemma \ref{lem:noncritical} we shall use a version of 
Lemma \ref{lem:bump} for a certain holomorphic map 
$K\times \overline \B^N\to \C^n$ for some $N\in\N$, 
where $\B^N$ is the unit ball of $\C^N$. We see by inspection that 
the same proof applies by using the Oka--Weil theorem instead of 
the Runge theorem. This will be used in the proof of 
Lemma \ref{lem:noncritical}.
\end{remark}

%
%
%    EXTENDING THE MAP ACROSS A BUMP
%
%
\section{Extending a holomorphic map across an annulus}\label{sec:noncritical} 

By using Lemma \ref{lem:bump} and gluing methods from Oka theory,
we now prove the following lemma, which is the main ingredient in the proof of 
Theorems \ref{th:main1} and \ref{th:main2}. It provides the so-called noncritical
case in the construction of holomorphic maps $R\to X$. 

%
%  FF   the statement and proof of this lemma have been adjusted to deal with the
%  exceptional set E. Putting this here, the proof of the main theorem goes more smoothly.
%

%
% Lemma: noncritical
%
\begin{lemma}\label{lem:noncritical}
Let $X$ be a complex manifold of dimension $n$ with a complete distance
function $\dist_X$. Assume that $\Omega\subset X$ is an open subset and
$E\subset \Omega$ is a closed subset with $\Hcal^{2n-1}(E)=0$ such that 
$X$ is dominable by a tube of lines at every point $x\in\Omega\setminus E$.  
Let $R$ be an open Riemann surface and $K\subset L$ be compact 
Runge sets in $R$ with piecewise smooth boundaries 
such that $K\subset \mathring L$ 
and $K$ is a strong deformation retract of $L$. 
Assume that $f:K \to X$ is a map of class $\Ascr(K)$ 
such that $f(bK)\subset \Omega$. 
Given a finite set $A=\{a_1,\ldots,a_m\}\subset \mathring K$ 
and numbers $\epsilon>0$ and $k\in\N$, 
there is a holomorphic map $\tilde f:L\to X$ satisfying 
\begin{itemize}
\item[\rm (i)] $\sup_{p\in K} \dist_X (\tilde f(p),f(p)) <\epsilon$, and  
\item[\rm (ii)] $\tilde f$ agrees with $f$ to order $k$ at every point of $A$.
\end{itemize}
If $\Omega$ is dominable by a tube of lines at every point $x\in \Omega\setminus E$
then $\tilde f$ can be chosen such that %it also satisfies 
\begin{itemize}
\item[\rm (iii)] $\tilde f(L\setminus \mathring K)\subset \Omega$.
\end{itemize}
\end{lemma}

Note that the condition that $X$ be dominable at a point $x\in \Omega\subset X$ 
is weaker than the condition that $\Omega$ be dominable at $x$.
In the former case, a map % from a spanning tube of lines $T\subset\C^N$ $(N\ge \dim X)$, 
which is dominating at $x\in\Omega$ may have range in $X$, while in the latter case it must lie in $\Omega$.

\begin{proof}
The conditions imply that the compact set 
$D=L\setminus \mathring K$ is a union of finitely many pairwise disjoint annuli.
It suffices to consider the case when $D$ is a single annulus
since the same construction can be performed independently on each of them.

Before proceeding, we recall the following implication of the
implicit function theorem. Given a holomorphic map $\sigma:T\to X$ 
of complex manifolds which is a submersion at a point $t_0\in T$,
there are open neighbourhoods $t_0 \in T_0\subset T$, $\sigma(t_0)\in X_0\subset X$ and a biholomorphic map $\phi:T_0\to X_0\times \Delta^d$ with 
$\dim X +d=\dim T$ such that $\sigma|_{T_0} = \pi\circ \phi$, where
$\pi:X_0\times \Delta^d\to X_0$ is the projection $\pi(x,t)=x$.  
We shall call such a triple $(T_0,\sigma,X_0)$ a {\em submersion chart}.
(When $\dim T=\dim X$, this says that $\sigma$ is locally biholomorphically 
at a point of maximal rank.) Given a submersion chart
$(T_0,\sigma,X_0)$, every holomorphic map $f:Y\to X_0$ lifts to a holomorphic 
map $g:Y \to T_0$ such that $\sigma\circ g=f$.
\[
	\xymatrixcolsep{4pc}
	\xymatrix{& T_0 \ar[d]_{\sigma} \ar[r]^\phi & X_0\times \Delta^d \ar[d]^\pi \\
	Y\ar[r]^{f}\ar[ur]^{g} & X_0 \ar[r]^{\mathrm{Id}} & X_0}
\]

Given $l\in\N$ we write $\Z_l=\Z/l\Z = \{0,1,\ldots,l-1\}$. 
The assumptions on $E\subset \Omega\subset X$ and the general 
position argument in Remark \ref{rem:transversality}
imply that, after a small perturbation of $f$ which is fixed to the given 
order $k$ in the  points in $A\subset \mathring K$,
we have that $f(bK)\subset \Omega\setminus E$.
Hence, we can subdivide the closed Jordan curve $bK$ into a finite union of 
compact subarcs $\{\alpha_i: i\in \Z_l\}$ lying back-to-back and satisfying 
the following conditions.
\begin{enumerate}[\rm ({A}1)]
\item $\alpha_i$ and $\alpha_{i+1}$ have a common endpoint $p_{i+1}$ and 
are otherwise disjoint for every $i\in\Z_l$. 
\item $\bigcup_{i\in\Z_l} \alpha_i=bK$.
\item For every $i\in \Z_l$ there are a spanning tube of lines $T_i\subset \C^{n_i}$
for some $n_i \ge \dim X$, a holomorphic map $\sigma_i:T_i\to X$, a neighbourhood 
$U_i\subset \Omega\setminus E$ of $f(\alpha_i)$, 
and an open subset $\omega_i\subset T_i$ such that 
$\sigma_i(\omega_i)=U_i$ and the triple $(\omega_i,\sigma_i,U_i)$ is a submersion chart.
\end{enumerate}
Condition (A3) can be achieved if the arcs $\alpha_i$ are chosen sufficiently short.

Let $p_i$ and $p_{i+1}$ denote the endpoints of $\alpha_i$, ordered so that 
$p_{i+1}=\alpha_i\cap \alpha_{i+1}$ for each $i\in\Z_l$.
Choose an embedded arc $\gamma_i\subset D$ connecting the point 
$p_i$ to a point $q_i\in bL$ so that these arcs are pairwise disjoint,
they intersect $bD=bK\cup bL$ only in the  respective endpoints $p_i$ and $q_i$,
and these intersections are transverse. Note that the compact set 
\begin{equation} \label{eq:S}
	S=K\cup\bigcup_{i\in\Z_l}\gamma_i
\end{equation} 
is admissible for Mergelyan approximation. Recall that $f(\alpha_i)\subset U_i$ by 
condition (A3). We extend $f$ as a constant map to each arc $\gamma_i$ having 
the value $f(p_i)$, and we use Mergelyan theorem (see Theorem \ref{th:Mergelyan})
to approximate the resulting map $f:S\to X$ uniformly on $S$ 
by a holomorphic map $V\to X$ on a neighbourhood $V\subset R$ of $S$, 
which we still denote by $f$. Assuming that the approximation is close enough, we have that 
\begin{equation} \label{eq:falphai}
	f(\gamma_{i}\cup \alpha_i\cup \gamma_{i+1})\subset U_i
	\ \ \text{holds for each $i\in\Z_l$}.
\end{equation}
Let $\wt S\subset V$ be a thin compact neighbourhood of $S$ with smooth boundary and set 
\begin{equation} \label{eq:wtK}
	\wt K=L\cap \wt S \subset V.
\end{equation} 
In light of \eqref{eq:falphai} and $U_i\subset \Omega$ (see (A3)), 
the set $\wt S$ can be chosen such that 
\[ %begin{equation} \label{eq:fwtK}
	f(\wt K\setminus \mathring K) \subset \Omega\setminus E.
\] %end{equation}
Furthermore, we can choose $\wt S$ (and hence $\wt K$) such that the set 
\begin{equation} \label{eq:Di}
	\overline{L \setminus \wt K}=\bigcup_{i\in\Z_l} D_i
\end{equation}
is the union of pairwise disjoint compact discs $D_i$ with piecewise smooth 
boundaries, and for each $i\in \Z_l$ the arc 
$\wt \alpha_i = \overline{bD_i\cap \mathring L}$
is so close to the arc $\gamma_{i}\cup \alpha_i\cup \gamma_{i+1}$ 
that \eqref{eq:falphai} implies 
\begin{equation}\label{eq:tildealphai}
	f(\wt \alpha_i) \subset U_i\quad\text{for all $i\in\Z_l$}.
\end{equation}
The complementary arc $bD_i\setminus \wt \alpha_i$ lies in $bL$. 
See Figure \ref{fig:51}.
%
% Figure
%
\vspace{2mm}
\begin{figure}[ht]
	\includegraphics[width=.6\textwidth]{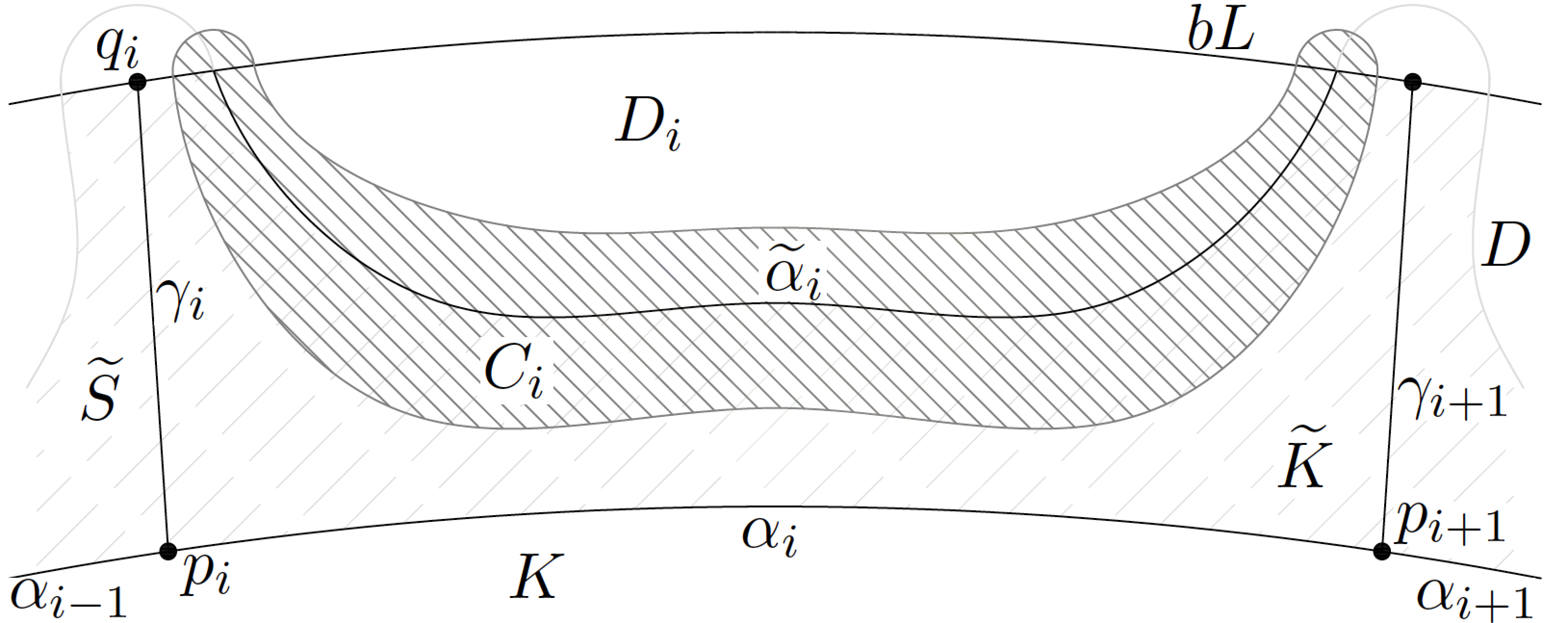}
	\caption{Sets is the proof of Lemma \ref{lem:noncritical}}\label{fig:51}
\end{figure}

By standard methods, using flows of holomorphic vector fields 
and up to shrinking the neighbourhood $V$ around $\wt K$ \eqref{eq:wtK} if necessary, 
we can find a holomorphic map 
\begin{equation}\label{eq:F} 
	F:V\times \B^N\to X\ \ \text{for some $N\ge \dim X$} 
\end{equation}
satisfying the following conditions (see \cite[Lemma 5.10.4]{Forstneric2017E}):
\begin{enumerate}[\rm (F1)]
\item $F(\cdotp,0)=f$, 
\item $F(\cdotp,t)$ agrees with $f$ to order $k$ 
at each point of $A\subset \mathring K$ for every $t\in\B^N$, and 
\item the partial derivative 
$\frac{\di}{\di t}\big|_{t=0}F(p,t):\C^N\to T_{f(p)}X$ is surjective for all $p\in V\setminus A$. 
\end{enumerate}
Such $F$ is called a (local) holomorphic spray with the core $f$ which is dominating  
on $V\setminus A$. (Here, $A=\{a_1,\ldots,a_m\}$ is the finite set given in the statement.)

For each $i\in\Z_l$ we pick a compact smoothly bounded disc neighbourhood 
$C_i\Subset V$ of the arc $\wt \alpha_i$ such that $f(C_i)\subset U_i$ 
(see \eqref{eq:tildealphai}) and $C_i\cap A=\varnothing$. Hence for some $r\in (0,1)$ we have 
that $F(C_i \times r\overline\B^N)\subset U_i$. Replacing the map $F(\cdotp,t)$ by 
$F(\cdotp,rt)$ for each $t\in\B^N$, we may assume that this holds for $r=1$. 
Furthermore, $C_i$ can be chosen such that 
$D_i\setminus \mathring C_i$ is a closed disc attached to $C_i$ as in 
Lemma \ref{lem:bump}; see Figure \ref{fig:51}.

Since $(\omega_i,\sigma_i,U_i)$ is a submersion chart for the map 
$\sigma_i:T_i\to X$ (see condition (A3)), there exists for each $i\in \Z_l$ a holomorphic map 
$g_i:C_i \times \overline \B^N \to \omega_i\subset T_i\subset \C^{n_i}$ such that  
\begin{equation}\label{eq:Tj}
	F = \sigma_i \circ g_i\ \ \text{holds on $C_i\times \overline \B^N$}.
\end{equation}
Since $D_i\setminus \mathring C_i$ is a closed disc attached to $C_i$ along
an arc, Lemma \ref{lem:bump} and Remark \ref{rem:bump} 
show that we can approximate the map $g_i$ as closely as desired 
uniformly on $C_i\times \overline \B^N$ by a holomorphic map 
$\tilde g_i:(C_i\cup D_i)\times \overline \B^N \to T_i\subset \C^{n_i}$. 
Define the map $G_i$ by
\begin{equation}\label{eq:Gi}
	G_i=\sigma_i\circ \tilde g_i:(C_i\cup D_i)\times \overline \B^N \to X
	\ \ \text{for $i\in\Z_l$}.
\end{equation} 
It follows from \eqref{eq:Tj} that $G_i$ approximates $F$ uniformly
on $C_i \times \overline \B^N$ for every $i\in\Z_l$. 

Recall that the spray $F$ is dominating on $V\setminus A$ 
and $C_i\cap A=\varnothing$ for every $i\in\Z_l$.
Hence, there is a number $r\in (0,1)$ depending on $F$ such that, if 
the approximations are close enough, we can glue $F$ and the sprays $G_i$ 
$(i\in\Z_l)$ into a holomorphic spray $\wt F:L\times r\B^N\to X$ which approximates 
$F$ on $\wt K\times r\B^N$,
and it agrees with $F$ (and hence with $f=F(\cdotp,0)$) to order $k$ at every point of 
$A\times r\B^N$ (see (F2)). We refer to 
\cite[Propositions 5.8.1 and 5.9.2]{Forstneric2017E} for this gluing technique.
The holomorphic map $\tilde f=\wt F(\cdotp,0):L\to X$ then 
satisfies the conclusion of Lemma \ref{lem:noncritical} provided that the
approximations made in the proof are close enough.

It remains to justify (iii). As before, we may assume after a small
perturbation of $f$ that $f(bK)\subset \Omega\setminus E$. There is a subdivision of $bK$
into arcs $\alpha_i$ as in (A2), open neighbourhoods $U_i\subset \Omega\setminus E$
of their images $f(\alpha_i)$, and holomorphic maps $\sigma_i:T_i\to \Omega$ from 
spanning tubes of lines $T_i\subset\C^{n_i}$ satisfying condition (A3). 
We can choose the set $\wt K$ in \eqref{eq:wtK}, neighbourhoods $C_i\subset V$ 
of the arcs $\wt \alpha_i = bD_i\cap L$ in \eqref{eq:tildealphai},
and the spray $F$ in \eqref{eq:F} such that $F(C_i\times \overline\B^N) \subset \Omega$
holds for every $i\in\Z_l$. From \eqref{eq:Gi} and $\sigma_i(T_i)\subset\Omega$ 
it follows that each spray $G_i$ \eqref{eq:Gi} has range in $\Omega$. 
An inspection of the gluing method 
(see \cite[Propositions 5.8.1 and 5.9.2]{Forstneric2017E}) shows that the 
spray $\wt F:L\times r\B^N\to X$ obtained by gluing $F$ with the $G_i$'s satisfies 
\[
	\wt F((C_i\cup D_i)\times \overline r\B^N)\subset 
	G_i((C_i\cup D_i)\times \overline \B^N) \subset\Omega
	\ \ \text{for every $i\in\Z_l$}.
\]
Since $L = \wt K \cup\, \bigcup_{i\in\Z_l} D_i$ (see \eqref{eq:Di}) and the holomorphic map 
$\tilde f=\wt F(\cdotp,0):L\to X$ approximates $f$ on $\wt K$, it follows that  
$\tilde f(L\setminus \mathring K)\subset \Omega$, so condition (iii) in the lemma holds.
\end{proof}

%
%
%  PROOF OF THEOREM 1.1
%
%
\section{Proof of Theorem \ref{th:main1}}\label{sec:proof}

In this section we prove the following result, which includes Theorem \ref{th:main1} 
as a special case with $\Omega=X$ (compare with Definition \ref{def:Oka1} of an
Oka-1 manifold). 

%
% FF   I noticed that Omega need not be connected, and I revised the proof accordingly.
%       In fact, different components of f(bK) may lie in different connected components 
%       of Omega. Lemma 5.1 also holds in this setting.
%

%
%  MAIN THEOREM 2
%
\begin{theorem}\label{th:main2}
Let $X$ be a complex manifold endowed with a complete distance function 
$\dist_X$, and let $\Omega \subset X$ be a domain which is  
densely dominable by tubes of lines (see Remark \ref{rem:dominable}). 
Assume that $R$ is an open Riemann surface, $K$ is a compact Runge set in $R$, 
$a_i \in R$ is a discrete sequence without repetitions, 
and $f:R\to X$ is a continuous map which is holomorphic on an open neighbourhood of 
$K\cup \bigcup_i\{a_i\}$ and satisfies $f(R\setminus\mathring K)\subset\Omega$.
Given $\epsilon>0$ and positive integers $k_i\in\N$, there is a holomorphic map $F:R\to X$ 
which is homotopic to $f$ and satisfies the following conditions:
\begin{enumerate}[\rm (i)]
\item $\sup_{p\in K}\dist_X(F(p),f(p))<\epsilon$, 
\item $F(R\setminus \mathring K) \subset \Omega$, and 
\item $F$ agrees with $f$ to order $k_i$ in the  point $a_i$ for every $i$.
\end{enumerate}
\end{theorem}

\begin{proof}
%
% FF   The argument in the next paragraph is not correct. The problem is that 
%         if we remove a thin subset E from Omega, the dominating maps from
%         tubes will no longer have values in this smaller Omega. We have to deal
%         with the issue at every step of induction.
%
\begin{comment}
By the argument in Remark \ref{rem:transversality}, we may assume without loss of 
generality that $\Omega$ is strongly dominable by tubes of lines.

Since $\Omega$ is assumed to be densely dominable by tubes of lines,
there is a closed subset $E\subset \Omega$ with $\Hcal^{2n-1}(E)=0$ such that
$\Omega$ is dominable by a tube of lines at every point of $\Omega\setminus E$.
By the argument in Remark \ref{rem:transversality}, a generic holomorphic 
map $g:M\to X$ from any compact bordered Riemann surface $M$ with 
$g(bM)\subset \Omega$ is such that $g(bM)\subset \Omega\setminus E$.
By applying this argument at every step of the proof to follow, we can 
replace $\Omega$ by $\Omega\setminus E$ and assume that...
% in the sense of Definition \ref{def:dominable-tubes}.
\end{comment}
%
Call $A=\{a_i\}_{i\in\N}$, and let $U\subset R$ be an 
open neighbourhood of $K\cup A$ such that the given map $f$ is holomorphic on $U$.
Pick a smooth strongly subharmonic Morse exhaustion function $\rho:R\to\R_+$ 
and an increasing sequence $0<c_1<c_2<\cdots$ diverging
to infinity such that the compact sets $K_j=\{\rho\le c_j\}$ for $j=1,2,\ldots$ 
satisfy the following conditions:
\begin{enumerate}[\rm (A)] 
\item $K\subset K_1\subset U$ and $A\cap K_1\subset K$.
\item The number $c_j$ is a regular value of $\rho$ 
and the smooth level set $\{\rho=c_j\}=bK_j$ is disjoint from $A$ for $j=1,2,\ldots$.
\item For every $j=1,2,\ldots$ the set 
$D_j=\{c_j\le \rho\le c_{j+1}\}=K_{j+1}\setminus \mathring K_j$
contains at most one critical point of $\rho$ or at most one point of $A$, but not both.
\end{enumerate}
The construction of such a function $\rho$ and a sequence $c_j$ is standard, 
using the fact that Morse critical points are isolated. 
Condition (A) is achieved by choosing $\rho$ and $c_1$ 
such that the set $K_1=\{\rho\le c_1\}$ is sufficiently close to $K$; this is possible since 
$K$ is Runge in $R$. Note that each $K_j=\{\rho\le c_j\}$ is a smoothly bounded Runge 
compact domain in $R$ and the sequence 
$K_1\Subset K_2\Subset \cdots \Subset \bigcup_{j=1}^\infty K_j = R$
is a normal exhaustion of $R$. 

Given a compact set $L\subset R$ and a pair of continuous maps $f,g:L\to X$, we write
\[
	d_L(f,g) = \max_{p\in L}\dist_X(f(p),g(p)). 
\]
%
% FF  For consistency I also introduced the quantities for j=0. 
%
Set $K_0=K$, $f_0=f|_K$, and 
$\epsilon_0=\epsilon/2$, where $K$ and $\epsilon>0$ are given in the statement. 
%
%  FF   I added the next part
% 
By the assumption, there is a closed subset $E\subset \Omega$ 
with $\Hcal^{2n-1}(E)=0$ such that $\Omega$ is dominable by a spanning tube 
of lines at every point of $\Omega\setminus E$.
By Remark \ref{rem:transversality}, there is a holomorphic map 
$f_1:K_1\to X$ which is $\epsilon_0$-close to $f$ on $K_1$, it agrees
with $f$ to order $k_i$ at every point $a_i\in A\cap K_{1}$, and it satisfies
\begin{equation}\label{eq:f1}
	f_1(K_1\setminus \mathring K_0)\subset\Omega\quad 
	\text{and}
	\quad f_1(bK_1)\subset \Omega\setminus E.
\end{equation}
Set $D_0=K_1\setminus \mathring K_0$. 
%
%  The end of addition
%
We shall inductively construct a sequence of holomorphic maps $f_j:K_j\to X$ 
and numbers $\epsilon_j>0$ satisfying the following conditions for each $j=1,2,\ldots$: 
\begin{enumerate}[\rm (a$_{j}$)] 
\item $d_{K_{j-1}}(f_{j},f_{j-1})<\epsilon_{j-1}$. 
\item $f_{j}(D_{j-1})\subset \Omega$
% FF   I added the next condition
and $f_j(bK_j) \subset \Omega\setminus E$. 
(Here, $D_{j-1}=K_{j}\setminus \mathring K_{j-1}$ is given by (C).) 
\item $f_{j}$ agrees with $f$ to order $k_i$ at every point $a_i\in A\cap K_{j}$.
\item $f_{j}$ is homotopic to $f|_{K_{j}}:K_{j}\to X$
%
% FF  the above is not precise enough, and I added the relevant condition.
%
by a homotopy mapping $K_j\setminus \mathring K$ to $\Omega$.
\item $\epsilon_{j}< \frac12\min\{\epsilon_{j-1},\dist_X(f_{j}(D_{j-1}),X\setminus \Omega)\}$.
% FF   I specified the following.
(If $\Omega=X$ then the second number under $\min$ is treated as $+\infty$.)
\end{enumerate}
%
% FF   The next few sentences were adjusted.
%
The beginning of the induction is given by the map $f_1$ found above
and a number $\epsilon_1$ satisfying (e$_{1}$).
Assume that for some $j\in\N$ we have found maps $f_k$ and numbers 
$\epsilon_k$ for $k=1,\ldots,j$, and let us explain
how to find the next pair $(f_{j+1},\epsilon_{j+1})$. We distinguish cases.

%
% CASE 1
%
\noindent {\em Case 1: $D_j$ does not contain any critical point of $\rho$ nor any point 
of $A$.} In this case, a map $f_{j+1}:K_{j+1}\to X$
satisfying (a$_{j+1}$)--(c$_{j+1}$) is given by Lemma \ref{lem:noncritical}.
%
% FF   I improved the next argument 
%
%Moreover, since $K_j$ is a strong deformation retract of $K_{j+1}$, we have that $f_{j+1}|_{K_j}$ is homotopic to $f_j$ provided the approximation in (a$_{j+1}$) is close enough, and hence to $f|_{K_j}$ by (d$_j$). Using again that $K_j$ is a strong deformation retract of $K_{j+1}$, this shows (d$_{j+1}$). 
Assuming as we may that the approximation in (a$_{j+1}$) is close enough, 
$f_{j+1}|_{K_j}$ is homotopic to $f_j$ by a homotopy staying close to $f_j$,
hence mapping $K_j\setminus \mathring K$ to $\Omega$. Since $K_j$ is a 
strong deformation retract of $K_{j+1}$ and $f(R\setminus \mathring K)\subset \Omega$, 
we obtain a homotopy from $f_{j+1}$ to $f$ satisfying (d$_{j+1}$). 
We then choose a number $\epsilon_{j+1}>0$ satisfying (e$_{j+1}$), 
thereby completing the induction step in this case.

%
% CASE 2
%
\noindent {\em Case 2: $D_j$ contains a critical point $p$ of $\rho$.} Our assumptions 
imply that such a point $p$ is unique, it is contained in the interior 
$\mathring D_j=\mathring K_{j+1}\setminus K_j$, and $D_j\cap A=\varnothing$.
There is a compact smoothly embedded arc 
$\lambda \subset bK_j\cup \mathring D_j = \mathring K_{j+1}\setminus \mathring K_j$
passing through $p$ and having endpoints in $bK_j$ such that $\lambda$ intersects $bK_j$
only in these endpoints and the intersections are transverse, and the set $K_j\cup\lambda$
is a strong deformation retract of $K_{j+1}$. (A discussion of the possible 
changes of topology depending on the Morse index of $\rho$ at $p$ can be found in 
\cite[Section 1.4]{AlarconForstnericLopez2021}.) 
% Since the set $\Omega\subset X$ is open and connected, 
% FF  Connectedness is irrelevant here. What we need is the homotopy as in (d_j).
%
Condition (d$_{j}$) implies that we can extend $f_j$ smoothly across the arc $\lambda$ 
such that $f_j(\lambda)\subset\Omega$ and $f_j:K_j\cup\lambda\to X$ 
is homotopic to $f|_{K_j\cup\lambda}$ by a homotopy mapping 
$(K_j\cup\lambda)\setminus \mathring K$ to $\Omega$.
We apply Mergelyan theorem (see Theorem \ref{th:Mergelyan})
to approximate $f_j$ on $K_j\cup\lambda$ by a holomorphic map $\tilde f_j:V\to X$ 
on a neighbourhood $V\subset K_{j+1}$ of $K_j\cup\lambda$ such that 
$\tilde f_j(bK_j\cup \lambda) \subset\Omega$ and $\tilde f_j$ agrees with $f_j$ 
to the given order $k_i$ at all points $a_i\in A\cap K_j$. 
Then, there are a compact smoothly bounded neighbourhood $K'\subset V$ 
of $K_j\cup\lambda$ such that $\tilde f_j(K'\setminus \mathring K_j)\subset\Omega$
and $K_{j+1}\setminus K'$ is a finite union of annuli.
Applying Lemma \ref{lem:noncritical} as in Case 1 to the pair of sets $K'\subset K_{j+1}$ 
and the map $\tilde f_j$ gives a holomorphic map $f_{j+1}:K_{j+1}\to X$ satisfying 
conditions (a$_{j+1}$)--(d$_{j+1}$). 
Finally, we pick a number $\epsilon_{j+1}$ satisfying condition (e$_{j+1}$).
 
%
% CASE 3
%
\noindent {\em Case 3: $D_j$ contains a point $a_i\in A$.} 
Our assumptions imply that such a point is unique and $D_j$ does not contain 
any critical point of $\rho$. Hence, $D_j$ is a finite union of annuli
and $K_j$ is a strong deformation retract of $K_{j+1}$. 
Let $U\subset R$ be the open set on which the initial  
continuous map $f:R\to X$ is holomorphic. Pick a closed disc 
$\Delta \subset U\cap \mathring D_j$ 
containing the point $a_i$ in its interior. Choose a smooth embedded arc 
$\lambda\subset bK_j\cup \mathring D_j\setminus \mathring \Delta$ 
connecting a point $p\in bK_{j}$ to a point $q\in b\Delta$ 
such that $\lambda$ intersects $bK_j$ and $b\Delta$ 
only at $p$ and $q$, respectively, and the intersections are transverse. 
The set $\wt K_j=K_j\cup \Delta\cup \lambda$ is then a strong deformation retract 
of $K_{j+1}$ and an admissible set for Mergelyan approximation. 
%
% FF  Again, connectedness is irrelevant, and the new argument is based on the 
%       existence of a suitable homotopy.
%
%Since $\Omega$ is connected and we have $f_j(bK_j) \subset \Omega$ and $f(a_i)\in\Omega$ by the assumption, 
%
Since $f(R\setminus \mathring K)\subset \Omega$, condition (d$_{j}$) implies 
that $f_j(bK_j)\subset \Omega$ and $f(a_i)\in \Omega$ lie in the same connected
component of $\Omega$. Hence we can extend 
$f_j$ from $K_j$ to a smooth map $\tilde f_j:\wt K_j\to X$ which equals $f$ 
on $\Delta$ and $\tilde f_j$ is homotopic to $f|_{\wt K_j}$ by a homotopy mapping 
$\wt K_j\setminus \mathring K$ to $\Omega$. 
We can now complete the proof as in Case 2, 
first approximating the extended map $\tilde f_j:\wt K_j\to X$ by a holomorphic
map on a neighbourhood of $\wt K_j$ 
%which agrees with $\tilde f_j$ to the given order $k_m$ at every point $a_m\in \wt K_j$ (in particular, at $a_i$)
and then applying Lemma \ref{lem:noncritical} 
to find the next holomorphic map $f_{j+1}:K_{j+1}\to X$ satisfying 
(a$_{j+1}$) --(d$_{j+1}$). Finally we choose $\epsilon_{j+1}>0$ satisfying (e$_{j+1}$).

This completes the induction step in all cases. The theorem now follows by 
verifying that the limit map $F=\lim_{j\to\infty}f_j:R\to X$ exists and satisfies
the stated conditions. This is an obvious consequence of conditions 
(a$_j$)--(e$_j$) whose verification is left to the reader.
\end{proof}

%
%  REMARK ABOUT PROPOSITION 2.8
%
\begin{remark}\label{rem:weakOka1}
It is obvious that the proof of Theorem \ref{th:main2} also gives Proposition
\ref{prop:weakOka1}. Indeed, Case 1 (the noncritical case) in the proof holds by the
Oka-1 property with approximation and interpolation in the  finitely many points
in the sublevel set $K_j=\{\rho\le c_j\}$, while the proof of Cases 2 and 3
only uses the Mergelyan theorem (see Theorem \ref{th:Mergelyan}). 
\end{remark}

%
%
%  SECTION: FUNCTORIAL PROPERTIES OF OKA-1 MANIFOLDS
%
%
\section{Functorial properties of Oka-1 manifolds}\label{sec:functorial}

In this section we study the behaviour of the Oka-1 property under certain natural 
operations in the category of complex manifolds, and its relationship to other flexibility 
properties of complex manifolds studied in the literature. 
A survey of these issues for the smaller class of Oka manifolds can be found in 
\cite[Chapter 7]{Forstneric2017E} and 
\cite{Forstneric2023Indag,ForstnericLarusson2014IMRN}.

It is obvious from the definition that an increasing union of Oka-1 manifolds
is an Oka-1 manifold. The proof of the following simple observation is left to the reader.

\begin{proposition}
The product $Z=X\times Y$ is an Oka-1 manifold if and only if $X$ and $Y$ 
are Oka-1 manifolds.
\end{proposition}

Next, we look at the following problem. 
Let $X$ and $Y$ be connected complex manifolds.
Under which conditions on a surjective holomorphic map $h:X\to Y$ 
is $X$ an Oka-1 manifold if (and only if) $Y$ is an Oka-1 manifold?

Our first result in this direction concerns unramified holomorphic covering projections.

%
%   Ascent of Oka-1 property in covering space
%
\begin{proposition}\label{prop:ascent}
Let $h:X\to Y$ be a holomorphic covering projection.
\begin{enumerate}[\rm (a)]
\item If $Y$ is an Oka-1 manifold, then $X$ is an Oka-1 manifold. 
\item If $X$ has the Oka-1 property for complex lines $\C\to X$, then so does $Y$.
\end{enumerate}
\end{proposition}

\begin{proof}
To prove part (a), assume that $Y$ is an Oka-1 manifold. Let $K\subset L=K\cup D$ be 
compact sets in an open Riemann surface $R$ as in Proposition \ref{prop:weakOka1}, 
and let $f:K\to X$ be a holomorphic map. Then, the projection 
$g=h\circ f:K\to Y$ (see \eqref{eq:lifting}) 
can be approximated uniformly on $K$ by holomorphic maps $\tilde g:L\to Y$, 
with interpolation in given finitely many points $a_1,\ldots, a_m\in K$. We may assume 
that there is at least one point $a_i$ in each connected component of $K$.
If the approximation is close enough then $\tilde g$ is homotopic to $g$ on $K$,
and hence it lifts to a unique holomorphic map $\tilde f:L\to X$ which agrees with $f$ 
at the points $a_1,\ldots, a_m$ and approximates $f$ on $K$.
Hence, $X$ is Oka-1 by Proposition \ref{prop:weakOka1}.

This argument fails in the opposite direction since a map $K\to Y$ need not lift 
to a map $K\to X$, unless the set $K$ is simply connected. 
In the latter case, every holomorphic map $K\to Y$ admits a holomorphic lift $K\to X$.
To prove (b), assume that $X$ has the Oka-1 property for complex lines $\C\to X$. 
Let $K\subset\C$ be a smoothly bounded Runge compact set, hence simply connected, 
and let $g:K\to Y$ be a holomorphic map. Since $K$ is simply connected, $g$ lifts 
to a holomorphic map $f:K\to X$. By the assumption we can approximate 
$f$ as closely as desired by a holomorphic map $\tilde f: \C\to X$ which 
agrees with $f$ to a given order $k$ in the points $a_1,\ldots, a_m\in K$. 
The holomorphic map $\tilde g:=h\circ \tilde f:\C\to Y$ 
then approximates $g$ on $K$ and agrees with $g$ to order $k$ 
in the points $a_1,\ldots, a_m$. Applying this argument 
inductively on an increasing sequence of discs  exhausting $\C$ gives (b).
\end{proof}

\begin{problem}\label{prob:coverings}
If $h:X\to Y$ is a holomorphic covering projection and $X$ is an Oka-1 manifold, 
is $Y$ an Oka-1 manifold?
\end{problem}

On the other hand, since a tube of lines is simply connected, we have the following 
observation concerning the sufficient condition for Oka-1 manifolds in Theorem \ref{th:main1}.

\begin{proposition}\label{prop:coverings}
If $h: X\to Y$ is a holomorphic covering space, then $X$ is dominable by tubes of lines
if and only $Y$ is dominable by tubes of lines. 
The same holds for dense and strong dominabilities by tubes of lines.
\end{proposition}

%
%  AN OKA-1 MANIFOLD IS STRONGLY LIOUVILLE
%
Recall that a complex manifold $X$ is said to be {\em Liouville} if it carries no nonconstant 
negative plurisubharmonic functions, and {\em strongly Liouville} if the universal 
covering space of $X$ is Liouville. As remarked in the introduction, every 
Oka-1 manifold is Liouville. Proposition \ref{prop:ascent} gives the 
following stronger conclusion. 

\begin{corollary}\label{cor:SL}
Every Oka-1 manifold is strongly Liouville.
\end{corollary}

Next, we show that the class of Oka-1 manifolds is invariant under 
Oka maps with connected fibres. 
We recall this notion, referring to \cite[Sect.\ 7.4]{Forstneric2017E} and
\cite[Sect.\ 3.6]{Forstneric2023Indag} for a more detailed presentation.

A holomorphic map $h:X\to Y$ of complex manifolds 
is said to enjoy the {\em parametric Oka property with approximation and interpolation}
(POPAI) if for every holomorphic map $g:S\to Y$ from a Stein manifold $S$, 
any continuous lifting $f_0:S\to X$ is homotopic through 
liftings of $g$ to a holomorphic lifting $f:S \to X$ as in the following diagram, 
\begin{equation}\label{eq:lifting}
	\xymatrixcolsep{5pc}
   	 \xymatrix{  & X \ar[d]^{h} \\ S \ar[r]^{\ \ g} \ar@{-->}[ur]^{f} & Y}
\end{equation}
with approximation on a compact $\Oscr(S)$-convex subset of $S$
and interpolation on a closed complex subvariety of $S$ on which $f_0$ is 
holomorphic. Furthermore, the analogous conditions must hold for families 
of maps $g_p:S\to Y$ depending continuously on  a parameter $p$ in a 
compact Hausdorff space; see
\cite[Definitions 7.4.1 and 7.4.7]{Forstneric2017E} for the details.

%
%   OKA MAPS
%
A holomorphic map $h:X\to Y$ is said to be an {\em Oka map} if it enjoys POPAI and 
is a Serre fibration (see \cite{Larusson2004,Forstneric2010CR} and 
\cite[Definition 7.4.7]{Forstneric2017E}). Assuming that $Y$ is connected, 
such a map is necessarily a surjective submersion and its fibres are Oka manifolds 
(see \cite[Proposition 3.14]{Forstneric2023Indag}). 
In particular, the constant map $X\to point$ 
is an Oka map if and only if $X$ is an Oka manifold. 
More generally, if $h:X\to Y$ is an Oka map then $X$ is an Oka manifold if and only if 
$Y$ is an Oka manifold (see \cite[Theorem 3.15]{Forstneric2023Indag}). 

% For Oka-1 manifolds we have the following result.

%
%  THEOREM UP-DOWN FOR OKA-1 MANIFOLDS
%
\begin{theorem} \label{th:updown}
Let $h:X\to Y$ be an Oka map between connected complex manifolds.
\begin{enumerate}[\rm (a)]
\item If $Y$ is an Oka-1 manifold then $X$ is an Oka-1 manifold.
\item If $X$ is an Oka-1 manifold and the homomorphism 
$h_*:\pi_1(X)\to \pi_1(Y)$ of fundamentals groups is surjective, 
then $Y$ is an Oka-1 manifold.
\item If $h:X\to Y$ is a holomorphic fibre bundle projection with a connected Oka fibre, 
then $X$ is an Oka-1 manifold if and only if $Y$ is an Oka-1 manifold.
\end{enumerate}
\end{theorem}

\begin{proof} 
Fix a pair of compact sets $K\subset L=K\cup D$ in an open Riemann surface $R$ 
as in Proposition \ref{prop:weakOka1}, where $D$ is a disc attached to $K$ along 
a boundary arc $\alpha=bD\cap bK$. Also, let $A$ be a finite subset of $K$
and $k\in \N$.

%\smallskip
\noindent {\em Proof of (a).} Let $f_0:K\to X$ be a holomorphic map.
Since $Y$ is an Oka-1 manifold, the holomorphic map $h\circ f_0:K\to Y$ 
can be approximated uniformly on $K$ by holomorphic maps $g:L\to Y$ 
with interpolation to order $k$ in the points of $A$. Assuming that the approximation 
is close enough, we see as in \cite[proof of Theorem 3.15]{Forstneric2023Indag} that there 
is a holomorphic map $f_1:K \to X$ which is uniformly close to $f_0$ on $K$, it agrees
with $f_0$ to order $k$ in the points of $A$, 
and it satisfies $h\circ f_1=g$ on $K$; i.e., $f_1$ is a lifting of $g$, 
see \eqref{eq:lifting}. (The construction of $f_1$ uses a holomorphic family 
of holomorphic retractions on the fibres of $h$, 
provided by \cite[Lemma 3.16]{Forstneric2023Indag}.)
Since $g:L \to Y$ is a holomorphic map, $f_1$ is a 
holomorphic lifting of $g$ over $K$, and $h:X\to Y$ is an Oka map, 
we can approximate $f_1$ (and hence $f_0$) uniformly on 
$K$ by holomorphic maps $f:L \to X$ satisfying $h\circ f=g$ such that 
$f$ agrees with $f_1$ to order $k$ in the points of $A$. 
Hence, Proposition \ref{prop:weakOka1} shows that $X$ is an Oka-1 manifold.  

%\smallskip
\noindent {\em Proof of (b).} 
Note that each connected component of $K$ has the homotopy type of a
finite bouquet of circles. The assumption that the homomorphism 
$h_*:\pi_1(X,x)\to \pi_1(Y,h(x))$ is surjective for some 
(and hence for all) $x\in X$ therefore implies that 
every holomorphic map $g:K\to Y$ lifts to a continuous map $f_0:K\to X$ such that 
$h\circ f_0=g$. Since $h$ is an Oka map, we can homotopically 
deform $f_0$ to a holomorphic map $f_1:K\to X$ satisfying $h\circ f_1=g$. 
Since $X$ is an Oka-1 manifold, we can approximate $f_1$ as closely as desired 
uniformly on $K$ by a holomorphic map $\tilde f:L\to X$ which agrees with $f_1$ 
to order $k$ in the points of $A$. The holomorphic map $\tilde g=h\circ \tilde f:L\to Y$ 
then approximates $g$ uniformly on $K$, and it agrees with $g$ to order 
$k$ in the  points of $A$. Hence, $Y$ is an Oka-1 manifold 
by Proposition \ref{prop:weakOka1}.

\smallskip
\noindent {\em Proof of (c).} 
This follows from (a) and (b) by noting that a holomorphic fibre bundle map $h:X\to Y$ 
is an Oka map if and only if its fibre is an Oka manifold 
\cite[Corollary 7.4.8 (i)]{Forstneric2017E}, and if the fibre of $h$ is connected 
then $h_*:\pi_1(X,x)\to \pi_1(Y,h(x))$ is surjective.
\end{proof}

%The conclusion in part (c) of Theorem \ref{th:updown} is false in general if the fibre is not Oka. For example, if $X=Y\times E$ where the manifold $E$ is hyperbolic, then $X$ is not Oka-1 even if $Y$ is.

%
%   OKA-1 MAP
%
We now introduce the notion of an Oka-1 map by analogy with Oka maps.

\begin{definition}\label{def:Oka1map}
A holomorphic map $h:X\to Y$ of complex manifolds is an {\em Oka-1 map} if
\begin{enumerate}[\rm (i)] 
\item $h$ is a Serre fibration, and 
\item given an open Riemann surface $R$, a holomorphic map $g:R\to Y$,
and a continuous lifting $f_0:R\to X$ of $g$ which is holomorphic on a neighbourhood of
a compact Runge subset $K\subset R$, we can deform $f_0$ through liftings
of $g$ to a holomorphic lifting $f:R\to X$ which approximates $f_0$ as closely
as desired on $K$ and agrees with $f_0$ to any given finite order in given finitely
many points of $K$ (see \eqref{eq:lifting}). 
\end{enumerate}
\end{definition}

Note that the constant map $X\to point$ is an Oka-1 map if and only if $X$ is an 
Oka-1 manifold. Obviously, every Oka map is also an Oka-1 map, but the converse fails
at least for maps with noncompact fibres. We have the following analogue
of \cite[Proposition 3.14]{Forstneric2023Indag}.

%
%  FF  A new proposition
%
\begin{proposition}\label{prop:Oka1map}
An Oka-1 map $h:X\to Y$ to a connected complex manifold $Y$  
is a surjective submersion and its fibres are Oka-1 manifolds. 
\end{proposition}

\begin{proof} 
It follows from the definition of an Oka-1 map that for every holomorphic disc 
$g:\Delta\to Y$ and point $x\in h^{-1}(g(0)) \in X$
there is a holomorphic lifting $f:\Delta\to X$ with $h\circ f=g$ and $f(0)=x$.
Hence, every tangent vector $v\in T_{g(0)}Y$ lies in the image of
the differential $dh_x:T_x X\to T_{g(0)}Y$, so $h$ is a submersion. 
Since any pair of points in $Y$ lie in the image 
of a holomorphic disc $\Delta\to Y$, $h$ is surjective. 
Hence, every fibre $h^{-1}(y)$ $(y\in Y)$ is 
a closed complex submanifold of $X$. Applying the definition of an Oka-1 map
to liftings of constant maps $R \to y\in Y$ shows that $h^{-1}(y)$ is an 
Oka-1 manifold for every $y\in Y$.
\end{proof}

An inspection of the proof of Theorem \ref{th:updown} shows that it holds 
if $h:X\to Y$ is an Oka-1 map, so we obtain the following corollary.

\begin{corollary}\label{cor:updown}
If $h:X\to Y$ is an Oka-1 map of connected complex manifolds then 
the conclusion of Theorem \ref{th:updown} holds.
\end{corollary}

However, we do not know the answer to the following question.

\begin{problem}
Let $h:X\to Y$ be a holomorphic fibre bundle whose fibre is an Oka-1 manifold.
Is $h$ an Oka-1 map?
\end{problem}

The proof of the analogous affirmative result for Oka manifolds and Oka maps 
(see \cite[Theorem 5.6.5]{Forstneric2017E} and 
\cite[Theorem 3.15]{Forstneric2023Indag}) does not apply.

One might wonder why is the conclusion of Theorem \ref{th:updown} (b) weaker for 
Oka-1 manifolds than for Oka manifolds, where the condition on the homotopy groups 
is unnecessary. The reason is that Oka manifolds are characterized by the 
convex approximation property (CAP), which refers to holomorphic maps from bounded 
convex sets in complex Euclidean spaces to the given manifold; 
see \cite[Definition 5.4.3 and Theorem 5.4.4]{Forstneric2017E}. 

We now introduce an approximation condition on a  
complex manifold which implies that it is Oka-1;
see Definition \ref{def:LSAP}. This is a version of the 
{\em convex approximation property}, CAP (see \cite[Definition 5.4.3]{Forstneric2017E}),
applied to dominating holomorphic sprays on discs with images in a 
given open subset of $X$. 
It encapsulates the condition which is needed to glue a pair of sprays 
in the proof of Lemma \ref{lem:noncritical}. The condition is invariant under 
Oka maps and under dominating holomorphic maps, but  
we do not know whether it characterizes Oka-1 manifolds.

Let us call a pair of compact topological discs $D\subset D'\subset \C$ 
a {\em special pair} if both discs have piecewise smooth boundaries 
and $D'\setminus \mathring D$ is a disc attached to $D$ along a boundary arc.
A holomorphic map $F:D \times \B^N\to X$ to a complex manifold $X$ 
is called a {\em holomorphic spray} of maps $D\to X$ with the core $f=F(\cdotp,0)$. 
Such a spray is said to be {\em dominating} if the partial derivative 
$\frac{\di}{\di t}\big|_{t=0}F(z,t):\C^N\to T_{f(z)}X$ is surjective for all $z\in D$. 
(Dominating sprays were used in the proof of Lemma \ref{lem:noncritical}, see \eqref{eq:F}.)  
 
%
%  CONDITION LSAP
%
\begin{definition}\label{def:LSAP}
A complex manifold $X$ has the {\em local spray approximation property}, 
LSAP, at point $x\in X$ if there is an open neighbourhood $V\subset X$ of $x$ satisfying
the following condition. Given a special pair of compact discs $D\subset D'\subset \C$ 
and a dominating holomorphic spray $F:D\times \B^N\to V$, there is a number 
$r=r(F) \in (0,1)$ such that $F$ can be approximated as closely as desired uniformly
on $D\times r\overline\B^N$ by holomorphic maps $G:D' \times r\overline\B^N\to X$.

A manifold $X$ has LSAP if the above condition holds at every point $x\in X$.
\end{definition}

\begin{remark}\label{rem:LSAP}
(A) The sprays in Definition \ref{def:LSAP} can either be defined over 
open neighbourhoods of the compact discs $D$ and $D'$, or else they could be 
continuous over the closed disc and holomorphic over its interior. Our applications 
of LSAP work in both cases since the gluing of sprays on Cartan pairs
works in both cases (see \cite[Section 5.9]{Forstneric2017E}). 

(B) Since every holomorphic spray $F:D \times \B^N\to X$ can be extended to a 
dominating spray by adding additional parameters, thereby  
increasing the dimension $N$ and shrinking the ball if necessary
(see \cite[Lemma 5.10.4]{Forstneric2017E} for a more general result),
LSAP is equivalent to the condition that every holomorphic spray of discs 
(not necessarily dominating) can be approximated by a spray on a bigger disc 
$D'$ as in Definition \ref{def:LSAP}. 
\end{remark}

An inspection of the proof of Lemma \ref{lem:noncritical} shows that 
if $X$ is dominable by a tube of lines (or by $\C^n$) at $x\in X$, 
then $X$ satisfies LSAP at $x$. 
Furthermore, if $X$ has LSAP at every point $x\in \Omega\setminus E\subset X$,
using the notation of Lemma \ref{lem:noncritical}, then the conclusion of the 
said lemma holds. This observation and the proofs of Theorems \ref{th:main2} and 
\ref{th:updown} imply the following results.

%
%  IMPLICATIONS OF LSAP. 
%
\begin{proposition}\label{prop:LSAP}
\begin{enumerate}[\rm (a)]
\item A spanning tube of complex lines in $\C^n$ has LSAP.
\item If a complex manifold $X$ satisfies LSAP at every point $x\in X\setminus E$ 
in the complement of a closed subset $E\subset X$ with $\Hcal^{2\dim X-1}(E)=0$, 
then $X$ is an Oka-1 manifold. In particular, a complex manifold with LSAP is an 
Oka-1 manifold.
\item A complex manifold which is densely dominable by manifolds having LSAP
is Oka-1. In particular, if $f:X\to Y$ is a surjective holomorphic submersion 
and $X$ is an LSAP manifold, then $Y$ is an LSAP manifold.
\item A holomorphic fibre bundle $X\to Y$ with an LSAP fibre is an Oka-1 map.
%\item If $h:X\to Y$ is a holomorphic fibre bundle with a connected Oka-1 fibre, then $X$ is an Oka-1 manifold if and only if $Y$ is an Oka-1 manifold.
\end{enumerate}
\end{proposition}

\begin{proof}
Part (a) follows by inspecting the proof of Lemma \ref{lem:bump}.
% (see Remark \ref{rem:bump}).
Parts (b) and (c) follow from the proof of Theorem \ref{th:main2},  
with condition LSAP replacing the use of Lemma \ref{lem:noncritical} as
explained above. 
%(In part (c) we may use different manifolds with LSAP to dominate $X$ at different points.) 
Part (d) follows from the proof of Theorem \ref{th:updown}.
The details are similar to the proof that a holomorphic fibre bundle
with Oka fibre is an Oka map (see \cite[Corollary 7.4.8]{Forstneric2017E}),
except that we also use localization as in the proof of Theorem \ref{th:main2}.
Note also that a lifting $f:D\to X$ of a holomorphic map $g:D\to Y$ in a holomorphic 
fibre bundle $h:X\to Y$ corresponds to a holomorphic section of the pullback bundle
$g^*X\to D$. 
\end{proof}

%
%	OKA-1 PROPERTY IN FAMILIES OF MANIFOLDS
%
Another interesting question is whether the set of Oka-1 manifolds is open or closed in 
a smooth family of complex manifolds. By \cite[Corollary 5]{ForstnericLarusson2014IMRN} 
(see also \cite[Corollary 7.3.3]{Forstneric2017E}), compact complex surfaces that are 
Oka (and hence Oka-1) can degenerate to a surface that is not strongly Liouville, 
and hence is not an Oka-1 manifold by Corollary \ref{cor:SL}. 
This shows that the property of being Oka-1
is not closed in families of compact complex manifolds. 
Concerning families of open manifolds, in \cite[Section 10]{Forstneric2023Indag} 
there is an example of a holomorphic submersion $h:X\to \C$ from a Stein 3-fold $X$  
such that $h$ is a trivial holomorphic fibre bundle 
with fibre $\C^2$ over $\C^*=\C\setminus\{0\}$, while the fibre $X_0$ over $0\in\C$
is the product $\Delta\times \C$ which is not Liouville, and hence not Oka-1.

There are immediate examples showing that the property of being Oka or Oka-1 is not
open in families of noncompact complex manifolds. For example, one can consider the 
family $h:X=\{(z,w)\in\C^2 : |zw|<1\}\to \C$ whose fibre over any $z\in \C^*$ is the disc, 
while the fibre over $z=0$ is $\C$. On the other hand, we are not aware of an example 
of an isolated Oka or Oka-1 fibre in a smooth family of compact complex manifolds. 

%
%
% SECTION: OKA-1 MANIFOLDS AMONG COMPACT COMPLEX SURFACES
%
%
\section{Oka-1 manifolds among compact complex surfaces}\label{sec:surfaces}
In this section, we summarize what we know about which compact 
complex surfaces are Oka-1.
Our discussion is based on the Enriques--Kodaira classification of such
surfaces (see Barth et al.\ \cite[Table 10, p.\ 244]{BarthHulek2004}), 
combined with Corollary \ref{cor:dominable} and 
the results of Buzzard and Lu \cite{BuzzardLu2000} on holomorphic dominability. 
The analogous analysis concerning compact Oka surfaces 
can be found in the paper
\cite{ForstnericLarusson2014IMRN} by Forstneri\v c and L\'arusson 
and in \cite[Section 7.3]{Forstneric2017E}. Comparing the two lists, 
we shall see that for many classes of compact complex surfaces
with Kodaira dimension $<2$ the properties of being Oka, Oka-1, 
and dominable by $\C^2$, are pairwise equivalent. 
The main exceptions are the K3 surfaces
and the elliptic fibrations. In the K3 class, Kummer surfaces and the 
elliptic K3 surfaces are Oka-1 (see Proposition \ref{prop:Kummer} and 
Corollary \ref{cor:K3elliptic}), but it is not known whether any or 
all of them are Oka manifolds.

The most important invariant of a compact complex manifold $X$ is 
its Kodaira dimension $\kappa_X\in \{-\infty,0,1,\ldots,n=\dim X\}$.
Let $K_X=\Lambda^n T^*X$ denote the canonical line bundle of $X$, and
for each integer $m\in\N$ let $P_m(X) = h^0(K_X^{\otimes m})$ 
denote the dimension of the complex vector space of holomorphic sections 
of the $m$-th tensor power of $K_X$.
Then, $k=\kappa_X$ is the integer such that $P_m(X)$ grows like 
$m^{k}$ as $m\to +\infty$, where $k=-\infty$ means that 
$K_X^{\otimes m}$ only admits the trivial (zero) section for every $m$. 
(See \cite[p.\ 29]{BarthHulek2004}.)
By Kodaira's pioneering work \cite{Kodaira1968} and its extensions 
(see Carlson and Griffiths \cite{CarlsonGriffiths1972}
and Kobayashi and Ochiai \cite{KobayashiOchiai1975}),
a compact complex manifold $X$ which is holomorphically or even just
meromorphically dominable by $\C^{\dim X}$ satisfies $\kappa_X<\dim X$. 
Manifolds with the maximal Kodaira dimension $\kappa_X=\dim X$ 
are said to be {\em of general type}, and they cannot be Oka since they are
not dominable by $\C^{\dim X}$. Conjecturally no such manifold is Oka-1 either, 
since it is believed that any holomorphic line $\C\to X$ in a manifold 
of general type is contained in a proper complex subvariety of $X$. 
This conjecture of Lang \cite{Lang1986} from 1986 seems to be still open.

In the sequel and unless stated otherwise, 
$X$ denotes a compact complex surface with Kodaira dimension
$\kappa\in\{-\infty,0,1\}$. A complete list of such surfaces, classified according 
to the value of $\kappa$, can be found in the monograph by 
Barth et al. \cite[Table 10, p.\ 244]{BarthHulek2004}. There are 10 classes,
and every compact complex surface has a minimal model (obtained by blowing down
all $-1$ rational curves) in exactly one of these classes. This minimal model is unique up to 
isomorphisms, except for surfaces with minimal models in the classes 1 and 3.

A {\em fibration} $f:X\to C$ of a complex surface $X$ onto a complex curve $C$ 
is a proper surjective holomorphic map with connected fibres. 
A fibration is {\em relatively minimal} if there are no $-1$ curves on any fibre.
(Any such curve is rational and can be blown down.)
A fibration is said to be {\em elliptic} if the general fibre $X_p=f^{-1}(p)$ 
is an elliptic curve (a complex torus). Such a surface can have any Kodaira dimension 
$\kappa\in \{-\infty,0,1\}$. A compact complex surface $X$ is an {\em elliptic surface} 
if it admits an elliptic fibration $X\to C$ onto an elliptic curve.  
In this connection, we recall (see \cite[Theorem 15.4, p.\ 127]{BarthHulek2004}) that if 
$X$ is a compact complex surface, $f:X\to C$ is a fibration without singular fibres,
and the curve $C$ is either $\CP^1$ or elliptic, then $f$ is a holomorphic fibre bundle.

%
%   KODAIRA DIMENSION MINUS INFINITY
%
%{\em Kodaira dimension $\kappa=-\infty$}: 
\subsection{Kodaira dimension $\kappa=-\infty$} \label{ss:Kodairaminusinfinity}
Rational surfaces (including nonminimal ones) are Oka, and hence Oka-1.  
A ruled surface is Oka-1 if and only if its base is $\CP^1$ or an elliptic curve. 
Theorem~\ref{th:classVII} covers surfaces of class VII if the 
global spherical shell conjecture holds true.

Let us go through the list and justify these claims.

Every smooth {\em rational surface} is obtained by repeatedly 
blowing up a minimal rational surface. The minimal rational surfaces are the projective 
plane $\CP^2$, which is Oka, and the Hirzebruch surfaces $\Sigma_r$ for 
$r \in\Z_+$. The latter are holomorphic $\CP^1$-bundles over $\CP^1$, 
so they are Oka by \cite[Theorem 5.6.5]{Forstneric2017E}. 
Repeated blowups preserve the Oka property for surfaces in this class 
by \cite[Proposition 6.4.6]{Forstneric2017E}, so nonminimal rational surfaces 
are also Oka.

A {\em ruled surface} is the total space of a holomorphic fibre bundle 
$X\to C$ with fibre $\CP^1$ over a compact curve $C$ 
(see \cite[p.\ 189]{BarthHulek2004}). 
By \cite[Theorem 5.6.5]{Forstneric2017E} such a surface $X$ is Oka if and only if the base 
curve $C$ is Oka, which holds if and only if $C$ is $\CP^1$ 
or a quotient of $\C$ (see \cite[Corollary 5.6.4]{Forstneric2017E}). 
By Theorem \ref{th:updown} (c), $X$ is Oka-1 if and only if $C$ is
$\CP^1$ or a quotient of $\C$. Note that minimal ruled surfaces over $\CP^1$ 
are just the Hirzebruch surfaces. 

{\em Class VII} comprises the 
nonalgebraic compact complex surfaces of Kodaira dimension $\kappa=-\infty$. 
Minimal surfaces of class VII fall into several mutually disjoint classes.  
For second Betti number $b_2=0$, we have {\em Hopf surfaces} and {\em Inoue surfaces}.  
For $b_2\geq 1$, there are {\em Enoki surfaces}, {\em Inoue-Hirzebruch surfaces}, 
and {\em intermediate surfaces}; together they form the class of {\em Kato surfaces}. 
Conjecturally, all surfaces with $\kappa=-\infty$ and  $b_2\geq 1$ 
admit a global spherical shell, i.e., a neighbourhood of the 3-sphere $S^3\subset \C^2$ 
holomorphically embedded into $X$ so that the complement is connected.
If this \textit{global spherical shell conjecture} holds true then every minimal surface 
of class VII with $b_2\geq 1$ is a Kato surface. The conjecture was proved 
in the cases $b_2\in \{1,2,3\}$ by Teleman in the respective papers
\cite{Teleman2005,Teleman2010,Teleman2018}.
Assuming that the global spherical shell conjecture holds true, 
the following result gives a complete description of Oka-1 surfaces 
in class VII. For the corresponding description of Oka surfaces in this class,
see \cite[Theorem 4]{ForstnericLarusson2014IMRN} or \cite[Theorem 7.3.2]{Forstneric2017E}. 

%
% SURFACES OF CLASS VII
%
\begin{theorem}  \label{th:classVII}
Minimal Hopf surfaces and minimal Enoki surfaces are Oka, and hence Oka-1. 
Inoue surfaces, Inoue--Hirzebruch surfaces, and intermediate surfaces, 
minimal or blown up, are not strongly Liouville, and hence not Oka or Oka-1. 
\end{theorem}

Recall that every Oka-1 manifold is strongly Liouville by Corollary \ref{cor:SL}.

In summary, we have the following corollary concerning surfaces with 
$\kappa_X=-\infty$.

\begin{corollary}\label{cor:minusinfty}
Modulo the global spherical shell conjecture, the following conditions 
are equivalent for every minimal compact complex surface $X$ with 
$\kappa_X=-\infty$:
\[
	\text{Oka} \ \Longleftrightarrow\  \text{dominable by $\C^2$}  
	\ \Longleftrightarrow\ \text{Oka-1} 
	\ \Longleftrightarrow\ \text{not strongly Liouville}.
\]
\end{corollary}

%
%   KODAIRA DIMENSION ZERO
%
\subsection{Kodaira dimension} $\kappa=0$. 
Bielliptic surfaces, Kodaira surfaces, 
and tori are Oka, and hence Oka-1. Elliptic and Kummer K3 surfaces are
Oka-1 manifolds, but we do not know whether any or all of them are Oka manifolds.
Again, we proceed case by case.
 
{\em Tori} (unramified quotients of $\C^2$) are complex homogeneous 
manifiolds,  hence Oka (see \cite[Proposition 5.6.1]{Forstneric2017E}) 
and therefore Oka-1. In fact, more can be said about compact tori. 

\begin{proposition}\label{prop:torus}
A complex surface $Y$ bimeromorphic to a compact torus is densely
dominable by $\C^2$, and hence an Oka-1 manifold.
\end{proposition}

\begin{proof} 
Let $\Gamma$ be a lattice of rank $4$ in $\C^2$, that is, a free abelian
subgroup $\Gamma\cong\Z^4$ of $\C^2$. Given finitely many points 
$P=\{p_1,\ldots, p_m\}$ in the torus $X=\C^2/\Gamma$, the complement 
$X\setminus P$ is universally covered by $\C^2\setminus \wt \Gamma$, where 
$\wt \Gamma=\bigcup_{i=1}^m (a_i+\Gamma)$ and 
$a_i\in \C^2$ are points mapped to $p_i$ under the quotient projection.
Buzzard and Lu showed in \cite[Proposition 4.1]{BuzzardLu2000} 
that the discrete set $\wt \Gamma$ is tame in $\C^2$. Hence, 
$\C^2\setminus \wt \Gamma$ is an Oka manifold by 
\cite[Proposition 5.6.17]{Forstneric2017E}, and its unramified quotient 
$X\setminus P$ is an Oka manifold by \cite[Proposition 5.6.3]{Forstneric2017E}. 
Any compact complex surface $Y$ which is bimeromorphic to $X$ 
admits a dominating holomorphic map from such a complement $X\setminus P$
with a Zariski dense image in $Y$, so $Y$ is densely dominable by $\C^2$. 
The conclusion now follows from Corollary \ref{cor:dominable} (a).
\end{proof}

Most tori are not elliptic. The elliptic compact $2$-tori form a 3-dimensional
family in the 4-dimensional family of tori. A generic 2-torus does not contain
any compact complex curves. 

According to \cite[p.\ 245]{BarthHulek2004}, every {\em bielliptic surface}    
and every {\em primary Kodaira surface} is the total space of a 
holomorphic fibre bundle with torus fibre over a torus, so it is Oka by 
\cite[Theorem 5.6.5]{Forstneric2017E} . 
A {\em secondary Kodaira surface} is a proper unramified holomorphic quotient of 
a primary Kodaira surface, so it is Oka by \cite[Proposition 5.6.3]{Forstneric2017E}. 
They are elliptic fibrations over $\CP^1$ with $b_1(X)=1$ and with nontrivial canonical bundle.

A {\em K3 surface} is a simply connected surface ($b_1=0$) with trivial canonical bundle,
hence $\kappa=0$. We refer to Barth et al.\ \cite[Chapter VIII]{BarthHulek2004} and 
Huybrechts \cite{Huybrechts2016} for a detailed treatment of such surfaces;
the basic description in \cite[Section 4.2]{BuzzardLu2000} will suffice for our needs. 
The class of K3 surfaces includes Kummer surfaces, which form a dense codimension 
16 family in the moduli space of K3 surfaces, and the elliptic K3 surfaces, which form
a dense codimension one family in the moduli space. All elliptic fibrations in the K3 class 
are ramified. It is not known whether any or all K3 surfaces are Oka.
Here we will show that every Kummer surface and every elliptic K3 surface
is an Oka-1 manifold.

% By \cite[Corollary 3]{ForstnericLarusson2014IMRN} (see also \cite[Corollary 7.2.3]{Forstneric2017E}), every Kummer surface is strongly dominable. In view of Theorem \ref{th:main}, this implies that every Kummer surface is an Oka-1 manifold.

Let us recall the structure of Kummer surfaces; see \cite[p.\ 224]{BarthHulek2004}. 
Let $X=\C^2/\Gamma$ be a compact complex 2-torus, and let $h:\C^2\to X$ 
be the quotient (covering) map.  The involution $\C^2\to \C^2$, 
$(z_1,z_2)\mapsto (-z_1,-z_2)$ descends to an involution $\sigma$ on $X$ 
with 16 fixed points $P=\{p_1,\ldots,p_{16}\}$. 
(If $\omega_1,\ldots,\omega_4\in\C^2$ are the generators of the lattice $\Gamma$ 
then $p_1,\ldots,p_{16}$ are the images under $h$ of the 16 points 
$c_1\omega_1+\cdots+c_4\omega_4$, where $c_1,\ldots,c_4\in\{0,\frac{1}{2}\}$.)   
The quotient space $X/\{1,\sigma\}$ is a 2-dimensional complex space with 16 
singular points $q_1,\ldots,q_{16}$. Blowing up $X/\{1,\sigma\}$ at each 
of these points yields a smooth compact surface $Y$ containing pairwise disjoint 
smooth rational curves $C_1,\ldots,C_{16}$. (Each of them is a $-2$ curve.)   
This is the Kummer surface associated to the rank four lattice $\Gamma\subset\C^2$.

Let $C=\bigcup_{i=1}^{16} C_i$. Note that 
$Y\setminus C$ is an unramified two-sheeted quotient of $X\setminus P$.
We have seen in the proof of Proposition \ref{prop:torus} that $Y\setminus C$
and $X\setminus P$ are Oka manifolds. If $B$ is a finite subset 
of $Y\setminus C$ and $A$ is its saturated preimage in $X\setminus P$,
then the Zariski domains $Y_0=Y\setminus (B\cup C)$ and $X_0=X\setminus (A\cup P)$
are manifolds of the same kind, hence Oka. If $\wt Y$ is a compact surface bimeromorphic 
to $Y$, it admits a dominating holomorphic map from such 
$Y_0$ (for some $B$) with a Zariski dense image in $\wt Y$, so $\wt Y$ 
is densely dominable by $\C^2$. This proves the following statement.

%
%  KUMMER SURFACES ARE OKA
%
\begin{proposition}\label{prop:Kummer}
A compact complex surface bimeromorphic to a Kummer surface is densely 
dominable by $\C^2$, and hence an Oka-1 manifold.
\end{proposition}

Next, we consider elliptic fibrations. Suppose that $f:X\to C$ is an elliptic fibration 
over a complex curve $C=\overline C\setminus P$ obtained by removing at most 
finitely many points $P$ from a compact complex curve $\overline C$. 
(Such a curve $C$ is quasi-projective.) Let $m\ge 0$ denote the cardinality
of the set $P$. The Euler characteristic of $C$ equals 
\begin{equation}\label{eq:chiC}
	\chi(C) = \chi(\overline C)-m = 2-2g(\overline C) - m,
\end{equation}
where $g$ denotes the genus. We assume that the fibration $f:X\to C$ has 
at most finitely many multiple and singular fibres. Let $n_p\in \N$ for $p\in C$
denote the multiplicity of the fibre $X_p=f^{-1}(p)$, so $n_p=1$ means that
the fibre is not multiple, although it may still be singular. 
(See \cite[Chapter V]{BarthHulek2004} or \cite[Sect.\ 3.2]{BuzzardLu2000}
for a detailed description of this notion.)
The fibration $f:X\to C$ determines the divisor 
\begin{equation}\label{eq:D}
	D=\sum_{p\in C}\left(1-\frac{1}{n_p}\right)p
\end{equation} 
with rational coefficients of degree 
\begin{equation}\label{eq:degD}
	\deg D = \sum_{p\in C}\left(1-\frac{1}{n_p}\right)\in \mathbb Q_+.
\end{equation} 
The sum is over $p\in C$ with $n_p > 1$, and
$\deg D=0$ if and only of there are no multiple fibres. 

We can now state our main result concerning Oka-1 manifolds among
elliptic fibrations. 

%
%  STRONG DOMINABILITY OF ELLIPTIC FIBRATIONS
%
\begin{theorem}\label{th:elliptic} 
Assume that $C$ is a quasi-projective complex curve and $f:X\to C$ 
is a relatively minimal elliptic fibration with at most a finite number of multiple fibres.
Let $D$ denote the associated $\mathbb Q$-divisor \eqref{eq:D}. Then,
the following conditions are equivalent.
\begin{enumerate}[\rm (a)]
\item The surface $X$ is densely dominable by $\C^2$, and hence an Oka-1 manifold.
\item $\chi(C,D)=\chi(C) - \deg D \ge 0$. (See \eqref{eq:chiC} and \eqref{eq:degD}.)
\item There is a holomorphic map $\C\to X$ with a Zariski dense image.
\end{enumerate}
Assuming that $X$ is not bimeromorphic to a primary or a secondary Kodaira
surface, the above conditions are also equivalent to the following:
\begin{enumerate}[\rm (d)]
\item  The fundamental group $\pi_1(X)$ is a finite extension of an abelian group.
\end{enumerate}
\end{theorem}

\begin{proof}
This result is analogous to \cite[Theorem 3.9]{BuzzardLu2000} by Buzzard and Lu, 
except that condition (a) in their theorem, that $X$ be dominable by $\C^2$,  
is now replaced by the stronger condition that $X$ be densely dominable. 
Hence, it follows that for such fibrations these two conditions are equivalent.
Condition (d) is equivalent to dominability of $X$ by \cite[Theorem 3.23]{BuzzardLu2000}. 
To prove the theorem, it thus suffices to show the implication (b) $\ \Rightarrow\ $(a).
%We shall do this by inspecting \cite[proof of Theorem 3.9]{BuzzardLu2000}.

The support $P\subset C$ of the divisor $D$ \eqref{eq:D} 
is a finite subset of $C$ by the assumption. 
By \cite[IV 9.12]{FarkasKra1992} the pair $(C,D)$ has a uniformizing orbifold
covering $\alpha:\wt C\to C$ where $\wt C$ is one of the Riemann
surfaces $\CP^1,\C,\D$ according to $\chi(C,D)>0$,  $\chi(C,D)=0$, 
or $\chi(C,D)<0$. (Here, $\D$ is the unit disc in $\C$.) This means that $\alpha:\wt C\to C$
is a surjective branched holomorphic map such that 
$\alpha:\wt C\setminus \alpha^{-1}(P)\to C\setminus P$ is a finite covering map,
and for each $p\in P$, $\alpha$ has ramification index $n_p$ at every point
of the fibre $\alpha^{-1}(p)$. Then, the pullback elliptic fibration 
$\tilde f:\wt X=\alpha^*X\to \wt C$ has no multiple fibres, and the natural map
$\wt X\to X$ covering $\alpha$ is an unramified covering map.
Since the properties of being dominable or densely dominable by $\C^2$ are 
invariant under covering maps, this reduces the problem to the case when 
$C\in\{\CP^1,\C,\D\}$ and $X\to C$ is an elliptic fibration without multiple fibres 
(see \cite[Proposition 3.4]{BuzzardLu2000}.) The case 
$C=\D$ is excluded by condition (b) 
(see \cite[proof of Theorem 3.9]{BuzzardLu2000}), and the case $C=\CP^1$ 
reduces to $C=\C$ by removing a point. 

It remains to show that the total space $X$ of an elliptic fibration $f:X\to \C$
without multiple fibres is densely dominable. By \cite[Lemma 3.8]{BuzzardLu2000} 
there exists a holomorphic section $\sigma:\C\to X$ of the fibration $f$.
(The proof uses the fact that every singular fibre $X_p=f^{-1}(p)$ 
which is not a multiple fibre admits an irreducible component of multiplicity one,
so there is a local holomorphic section of $X$ at every point. 
This is seen by inspecting Kodaira's list of non-multiple singular fibres,
see \cite[Table 3, p.\ 201]{BarthHulek2004}. A global section is then found
by solving a Cousin-1 problem.) Let $X^\sigma_p$ denote the union of all 
irreducible components of the fibre $X_p$ which do not contain the point $\sigma(p)$
(such exist only if $X_p$ is a singular fibre). Their union 
$X^\sigma=\bigcup_p X^\sigma_p$ is a closed one-dimensional complex subvariety 
of $X$. By a theorem of Kodaira \cite[Proposition V-9.1, p.\ 206]{BarthHulek2004}
there is a canonical fibre-preserving isomorphism 
$\Theta:\mathrm{Jac}(f)\to\Omega \subset X$ from the Jacobian fibration
$\mathrm{Jac}(f)\to \C$ onto the Zariski open domain 
$\Omega = X\setminus X^\sigma$ in $X$. Recall that the fibre of the 
Jacobian fibration over $p\in \C$ is 
\[
	\mathrm{Jac}(f)_p=\mathrm{Pic}^0(X_p)=H^1(X_p,\Oscr_{X_p}) / H^1(X_p,\Z),
\]
where the inclusion $H^1(X_p,\Z)\hra H^1(X_p,\Oscr_{X_p})$ comes from the 
cohomology sequence of the exponential sheaf sequence 
$0\to \Z\hra\C\to \C^*\to 0$: 
\[
	0 \to H^1(X_p,\Z) \to H^1(X_p,\Oscr_{X_p}) \to 
	H^1(X_p,\Oscr_{X_p}^*)=\Pic(X_p) \to H^2(X_p,\Z)\to 0
\]
(see \cite[p.\ 627]{BuzzardLu2000}). Thus, $\mathrm{Pic}^0(X_p)$
is the subgroup of the group $\Pic(X_p)$ consisting of holomorphic line bundles 
$E\to X_p$ with trivial first Chern class $0=c_1(E)\in H^2(X_p,\Z)$.  
We have that $L_p:=H^1(X_p,\Oscr_{X_p}) \cong\C$ for all $p\in\C$,
$L=\bigsqcup_{p\in \C}L_p\to \C$ is a holomorphic line bundle, which is trivial
by the Oka--Grauert principle, and $\mathrm{Jac}(f)$ is a quotient of $L$.
Let $A\subset \C$ be the discrete set of points $p\in \C$ for which the fibre $X_p$ 
is singular, and set $X'=\bigcup_{p\in A}X_p$ and $L'=\bigcup_{p\in A}L_p$.
The natural quotient projection $\phi: L\to \mathrm{Jac}(f)$ is a 
nonramified covering map on the Zariski open set $L\setminus L'$ in $L$. The 
holomorphic map 
\[ %begin{equation}\label{eq:h8}
	h=\Theta\circ \phi: L\to \Omega = X\setminus X^\sigma
\] %end{equation}
is then surjective, and $h:L\setminus L' \to \Omega\setminus X'$ 
is a nonramified covering projection. Hence, $h$ is dominating 
on the Zariski open domain $\Omega\setminus X'$, so $X$ is densely 
dominable by $L\cong \C^2$.
\end{proof}

Since a K3 surface has trivial fundamental group, 
we have the following corollary to Theorem \ref{th:elliptic} 
(see part (d) of the theorem).

%
%  ELLIPTIC K3 SURFACES ARE OKA
%
\begin{corollary}\label{cor:K3elliptic}
Every elliptic K3 surface is densely dominable by $\C^2$, and hence Oka-1.
\end{corollary}

By the Enriques--Kodaira classification, every compact complex surface with 
Kodaira dimension $0$, which is not bimeromorphic to a 
torus or a K3 surface, is an elliptic fibration, so the question of their
(dense) dominability by $\C^2$ is covered by Theorem \ref{th:elliptic}. 

%
%  KODAIRA DIMENSION 1
%
%\smallskip {\em Kodaira dimension $\kappa=1$}:  
\subsection{Kodaira dimension $\kappa=1$}
%By the Enriques--Kodaira classification, 
Every such surface is an elliptic surface,
given as the total space of an elliptic fibration $X\to C$ over an elliptic curve.
These are called {\em properly elliptic surfaces.} 
The equivalence of (a) and (b) in Theorem \ref{th:elliptic} gives the following corollary.

%
%  PROPERLY ELLIPTIC SURFACES
%
\begin{corollary}\label{cor:properly elliptic}
Let $f:X\to C$ be an elliptic fibration over an elliptic curve $C$. Then, the elliptic
surface $X$ is densely dominable by $\C^2$ if and only if the fibration $f$ has no 
multiple fibres. Every elliptic surface with this property is an Oka-1 manifold.
\end{corollary}

\begin{proof}
Let $D$ be the $\mathbb{Q}$-divisor \eqref{eq:D}. Since $\chi(C)=0$, we have that 
\[
	\chi(C,D)=\chi(C)-\deg D=-\deg D,
\] 
so $\chi(C,D)\ge 0$ (condition (b) in Theorem \ref{th:elliptic}) holds if and only if 
$D=0$, which means that there are no multiple fibres.
\end{proof}

%
%
%  SECTION: RATIONALLY CONNECTED MANIFOLDS ARE OKA-1
%
% AA: I change the title of the section
%
\section{A conjecture on rationally connected manifolds}\label{sec:RC}

A complex manifold $X$ is said to be rationally connected 
if any pair of points in $X$ is connected by a rational curve $\CP^1\to X$. 
For results on this class of manifolds, we refer to the papers by Kollar et al.\ 
\cite{Kollar1991,KollarMiyaokaMori1992} and the monographs by 
Koll\'ar \cite{Kollar1995E} and Debarre \cite{Debarre2001}.
By \cite[Theorem 2.1]{KollarMiyaokaMori1992}, several possible definitions 
of this class coincide. In particular, if every sufficiently general pair of points 
in $X$ can be connected by an irreducible rational curve, then 
$X$ is rationally connected. 

There are reasons to believe that the 
following conjecture holds true.

%
%   RATIONALLY CONNECTED MANIFOLDS 
%
\begin{conjecture}\label{con:RC}
Every rationally connected projective manifold is an Oka-1 manifold.
\end{conjecture}

One indication is provided by the theorem of Campana and Winkelmann
\cite[Theorem 5.2]{CampanaWinkelmann2023} saying that every rationally 
connected projective manifold $X$ admits an entire holomorphic 
curve $\C\to X$ whose image contains a given countable subset of $X$, 
with prescribed finite order jets in these points 
\cite[Corollary 5.7]{CampanaWinkelmann2023}. 
In particular, $X$ contains dense entire curves $\C\to X$. 
Their construction relies on smoothing a comb or a tree of rational curves; 
see \cite[Proposition 5.1 and Lemma 5.5]{CampanaWinkelmann2023}.
This technique goes back to the seminal paper
\cite{KollarMiyaokaMori1992} by Koll\'ar, Miyaoka, and Mori; 
see also Koll\'ar \cite[Theorem 7.6, p.\ 155]{Kollar1995E}.

%
% Franc: added on April 9, 2025
%
A partial affirmative answer to Conjecture \ref{con:RC}
was given by Benoist and Wittenberg in 
\cite[Theorem 1.2]{BenoistWittenberg2025}.
They proved that every rationally simply connected projective manifold 
has the {\em tight approximation property}, 
a condition introduced in \cite{BenoistWittenberg2021} 
which means that holomorphic
maps from open subsets of compact Riemann surfaces 
to the given algebraic manifold can be approximated uniformly on 
compacts by regular algebraic maps. 
When the target manifold is compact, 
the tight approximation property is equivalent to the 
aOka-1 property, which is the Runge approximation property by
regular algebraic maps from finitely punctured compact Riemann surfaces
(see \cite[Definition 1.5]{ForstnericLarusson2025MZ}).
Clearly, such a map extends across the punctures if the target manifold
is compact. Since the aOka-1 property implies the Oka-1 property 
(see \cite[Proposition 1.9]{ForstnericLarusson2025MZ}),
it follows from \cite[Theorem 1.2]{BenoistWittenberg2025} 
that every rationally simply connected projective manifold
is aOka-1 and Oka-1. This holds 
in particular for all smooth hypersurfaces 
of degree $d$ in $\CP^n$ with ${n\geq d^2-1}$
(see \cite[Corollary 1.3]{BenoistWittenberg2025}). 
 
%
% Franc: the next paragraph has been edited.
%
A related result of Gournay \cite[Theorem 1.1.1]{Gournay2012} gives 
Runge approximation of almost holomorphic maps from 
%open subsets of 
compact Riemann surfaces to certain compact almost complex manifolds.
%by globally defined almost holomorphic maps.
For an algebraic target manifold, this agrees 
with the tight approximation property of Benoist and Wittenberg
mentioned above. The conditions on the target manifold
in Gournay's theorem hold true 
for all rationally connected projective manifolds. 
Hence, using his theorem at face value, 
it would follow that every such 
manifold enjoys the Oka-1 property with approximation. 
In the earlier version of our preprint \cite[Sect.\ 9]{AlarconForstneric2023Oka1}
it is also described how jet interpolation in finitely many points 
could be added, thereby verifying Conjecture \ref{con:RC}.
Unfortunately, we were unable to understand the proof of 
\cite[Theorem 1.1.1]{Gournay2012}.

%
%
% SECTION: DENSE HOLOMORPHIC CURVES
%
%
\section{Dense holomorphic curves in an arbitrary complex manifold}\label{sec:dense}

In this section we prove the following result which generalizes the case
$M=\Delta$ obtained by Forstneri\v c and Winkelmann in \cite{ForstnericWinkelmann2005MRL}. In the last part of the section,
we give an analogous result for holomorphic Legendrian curves
in complex contact manifolds; see Theorem \ref{th:denseLegendrian}.

%
% Theorem
%
\begin{theorem}\label{th:dense}
Assume that $X$ is a connected complex manifold endowed 
with a complete distance function $\dist_X$, $\overline M=M\cup bM$ 
is a compact bordered Riemann surface, and $f:\overline M\to X$ is 
a map of class $\Ascr(\overline M)$. Given a compact subset 
$K\subset M$, a countable set $B\subset X$, 
and a number $\epsilon>0$ there is a 
holomorphic map $F: M\to X$ such that 
\begin{enumerate}[\rm (a)]
\item $\sup_{p\in K}\dist_X(F(p),f(p))<\epsilon$, 
\item $F$ agrees with $f$ to a given finite order at a given finite set 
of points $C\subset M$, and
\item $B\subset F(M)$.
\end{enumerate}
Furthermore, $F$ can be chosen to be an immersion if $\dim X>1$, 
and to be an injective immersion if $\dim X>2$, 
whenever condition {\rm (b)} allows it.
\end{theorem}

We record the following immediate corollary.

%
% Corollary
%
\begin{corollary}\label{co:denseBRS}
If $X$ is a complex manifold and $M$ is a bordered Riemann surface, then there is a 
holomorphic map $M\to X$, which can be chosen to be an immersion if $\dim X>1$
and an injective immersion if $\dim X>2$, whose image is everywhere dense in $X$. 
\end{corollary}

It was shown by Forn\ae ss and Stout 
\cite{FornaessStout1977AJM,FornaessStout1982AIF} that 
every connected complex manifold $X$ of dimension $n$ admits 
surjective holomorphic maps $\Delta^n\to X$ and $\B^n\to X$ from the polydisc and the 
ball in $\C^n$. Hence, to obtain density, it would suffice to prove Corollary \ref{co:denseBRS}
for maps $M\to\Delta^n$. However, there seems to be no particular advantage in this 
reduction, which furthermore does not give the approximation and general position 
result in Theorem \ref{th:dense}.

%
% Proof of Theorem 
%
\begin{proof}[Proof of Theorem \ref{th:dense}]
We assume without loss of generality that $M$ is a smoothly bounded compact domain 
without holes in an open Riemann surface $R$ (see Stout \cite[Theorem 8.1]{Stout1965TAMS})
and $f_0=f$ is holomorphic on $\overline M$ (see Theorem \ref{th:Mergelyan}).
Let $K$, $B$, and $\epsilon$ be as in the statement, and assume as we may that $K_0=K$ 
is a smoothly bounded compact domain which is a strong deformation retract of $M$,  
the given countable set $B=\{b_1,b_2,\ldots\}$ is infinite,  
$C\subset \mathring K_0$, and $f(C)\cap B=\varnothing$.
(Here, $C$ is the finite set for the interpolation condition in {\rm (b)}.) 
Pick a point $a_0\in \mathring K_0\setminus (C\cup f^{-1}(B))$ and set $b_0=f_0(a_0)$. 
Let $\epsilon_0=\epsilon/2$ and $K_{-1}=\varnothing$. 
We shall inductively construct a sequence of holomorphic maps $f_j:\overline M\to X$, 
smoothly bounded compact domains $K_j\subset M$, numbers $\epsilon_j>0$, 
and points $a_j\in M$, $j\in\N$, satisfying the following conditions.
\begin{enumerate}[\rm (1$_{j}$)]
\item $K_{j-1}\cup\{a_j\}\subset \mathring K_j$ and $K_j$ is a strong deformation 
retract of $M$.
 
\item $\epsilon_j<\epsilon_{j-1}/2$.
 
\item $\sup_{p\in K_{j-1}}\dist_X(f_j(p),f_{j-1}(p))<\epsilon_j$.
 
\item $f_j(a_i)=b_i$ for all $i\in\{0,\ldots,j\}$.
 
\item $f_j$ agrees with $f$ to a given finite order at every point in $C$.
\end{enumerate}
In addition, we shall guarantee that
\begin{equation}\label{eq:cupK_jM}
	\bigcup_{j\in\N} K_j=M.
\end{equation}

Assume that we have such a sequence in hand. 
Conditions {\rm (1$_j$)}, {\rm (2$_j$)}, {\rm (3$_j$)}, 
and \eqref{eq:cupK_jM} ensure that there is a limit holomorphic map 
$F=\lim_{j\to\infty}f_j:M\to X$ satisfying condition (a) 
in the statement of the theorem. Moreover, conditions {\rm (5$_j$)} 
and {\rm (4$_j$)} guarantee (b) and (c), respectively. 
So, the map $F$ satisfies the conclusion of the theorem. 
Note that the final statement 
can be granted as explained in Remark \ref{rem:transversality}; 
it suffices to take the number $\epsilon_j>0$ sufficiently small at each step 
of the inductive construction.

It thus remains to explain the induction. The basis is provided by $f_0=f$, $K_0=K$, 
$\epsilon_0=\epsilon/2$, and $a_0\in \mathring K_0$; note that conditions {\rm (2$_0$)} 
and {\rm (3$_0$)} are void. For the inductive step, fix $j\in\N$ and assume that we have 
objects $f_{j-1}$, $K_{j-1}$, $\epsilon_{j-1}$, and $a_{j-1}$ satisfying conditions 
{\rm (1$_{j-1}$)}, {\rm (4$_{j-1}$)}, and {\rm (5$_{j-1}$)}. 
Choose any $\epsilon_j>0$ meeting {\rm (2$_j$)}. If $b_j\in f_{j-1}(M)$ then we choose 
any point $a_j\in M$ with $f_{j-1}(a_j)=b_j$ and any smoothly bounded compact domain
$K_j\subset M$ satisfying {\rm (1$_j$)}, 
and set $f_j=f_{j-1}$. Assume on the contrary that $b_j\notin f_{j-1}(M)$. 
Choose a point $x\in bM\subset R$ %with $f_{j-1}(x)\neq b_j$ 
and attach to $\overline M$ a smooth embedded arc $\gamma\subset R$ such that 
$\gamma\cap \overline M=\{x\}$ and the intersection of $bM$ and $\gamma$ is 
transverse at $x$. Let $x'\in R\setminus\overline M$ denote the other endpoint of 
$\gamma$. Fix a point $y\in \gamma\setminus\{x,x'\}$ and extend $f_{j-1}$ to a 
smooth map $f_{j-1}:\overline M\cup\gamma\to X$ which is holomorphic on a
neighbourhood of $\overline M$ and satisfies $f_{j-1}(y)=b_j$. 
Theorem \ref{th:Mergelyan} furnishes a holomorphic map $g:V\to X$ on an 
open neighbourhood $V\subset R$ of $\overline M\cup\gamma$ such that
\begin{enumerate}[\rm (i)]
\item $\sup_{p\in \overline M\cup\gamma}\dist_X(g(p),f_{j-1}(p))<\epsilon_j/2$,
 
\item $g(a_i)=b_i$ for all $i\in\{0,\ldots,j-1\}$,
 
\item $g(y)=b_j$, and
 
\item $g$ agrees with $f$ to a given order at every point in $C$.
\end{enumerate}
For (ii) and (iv) take into account {\rm (4$_{j-1}$)} and {\rm (5$_{j-1}$)}.
%; recall that $C\subset \mathring K_0\subset M$. 
Now, \cite[Theorem 6.7.1]{AlarconForstnericLopez2021}
(see also \cite[Theorem 2.3]{ForstnericWold2009}) furnishes a conformal diffeomorphism 
$\phi: \overline M\to\phi(\overline M)\subset V$ such that %$\phi(x)=x'$, 
$\phi$ agrees with the identity map to a given order in the  points in the finite set 
$C\cup\{a_0,\ldots,a_{j-1}\}\subset M$, $\phi$ is arbitrarily close to the identity map 
on $K_{j-1}$, and we have that
\begin{equation}\label{eq:yphi(M)}
	y\in \phi(M).
\end{equation} 
In particular, $\phi$ can be chosen so that the map $f_j=g\circ\phi:\overline M\to X$ 
of class $\Ascr(\overline M)$ satisfies
\begin{enumerate}[\rm (I)]
\item $f_j(a_i)=g(a_i)$ for all $i\in\{0,\ldots,j-1\}$,
 
\item $b_j\in f_j(M)$ (see (iii) and \eqref{eq:yphi(M)}),
 
\item $\sup_{p\in K_{j-1}}\dist_X(f_j(p),g(p))<\epsilon_j/2$, and
 
\item $f_j$ agrees with $g$ to a given order at every point in $C$.
\end{enumerate}
By Theorem \ref{th:Mergelyan} we may assume that $f_j$ is holomorphic on $\overline M$. 
Finally, (II) enables us to choose a point $a_j\in M$ with $f_j(a_j)=b_j$ and a 
smoothly bounded compact domain $K_j\subset M$ satisfying {\rm (1$_j$)}. 
It is then clear in view of these choices and conditions (i)--(iv) and (I)--(IV) 
that $f_j$, $K_j$, $\epsilon_j$ and $a_j$ satisfy the requirements in {\rm (1$_j$)}--{\rm (5$_j$)}. 
This closes the induction and completes the proof of the theorem.
Note that \eqref{eq:cupK_jM} is guaranteed by choosing the compact 
domain $K_j\subset M$ sufficiently large at each step of the inductive construction.
\end{proof}

Combining the arguments in 
\cite[proofs of Theorems 1.6 and 1.8]{AlarconForstneric2024AMPA} 
with the proof of Theorem \ref{th:dense}, one may easily establish the analogous hitting 
results for some classes of surfaces of infinite topology, at the cost of losing the 
control on the complex structure. In particular, we have the following corollary. 

%
% Corollary
%
\begin{corollary}\label{co:dense-2}
Let $X$ be a connected complex manifold and $B\subset X$ be a countable subset. 
The following assertions hold.
\begin{enumerate}[\rm (i)]
\item On every compact Riemann surface $R$ there is a Cantor set $C$ whose complement 
$M=R\setminus C$ admits a holomorphic map $F:M\to X$ with $B\subset F(M)$.
\item On every open orientable smooth surface $S$ there is a complex 
structure $J$ such that the Riemann surface $M=(S,J)$ 
admits a holomorphic map $F:M\to X$ with $B\subset F(M)$.
\end{enumerate}
In both cases, the holomorphic map $F:M\to X$ can be chosen to be an immersion if 
$\dim X>1$ and an injective immersion if $\dim X>2$.
\end{corollary}

We leave the details of the proof to the reader.
In particular, in both cases (and up to a suitable choice of the Cantor set 
$C$ in assertion (i) and of the complex structure $J$ in assertion (ii)) 
there is a holomorphic map $M\to X$, which can be chosen an immersion 
if $\dim X>1$ and an injective immersion if $\dim X>2$, whose image is 
everywhere dense in $X$. 
 
The analogue of Theorem \ref{th:dense} also holds for holomorphic 
Legendrian immersions in an arbitrary connected complex 
contact manifold $(X,\xi)$. These are immersions
which are tangential to the holomorphic contact subbundle $\xi$ 
of the tangent bundle $TX$. The only essential difference in the proof 
is that we use the Mergelyan approximation theorem for Legendrian 
immersions $f:S\to X$ from an admissible set $S$ 
in an open Riemann surface $R$, 
given by \cite[Theorem 1.2]{Forstneric2022APDE}. 
We refer to that paper for the background on this subject. 
The precise result that one obtains is the following; 
we leave the details to the reader.

%
% Theorem
%
\begin{theorem}\label{th:denseLegendrian}
Assume that $(X,\xi)$ is a connected complex contact manifold,
$\overline M=M\cup bM$ is a compact bordered Riemann surface, 
and $f:\overline M\to X$ is a $\xi$-Legendrian immersion of class $\Ascr^r(\overline M)$ 
for some integer $r\ge 4$ (i.e., smooth of order $r$ on $\overline M$
and holomorphic on $M$). 
Given a compact subset $K\subset M$, a countable set $B\subset X$, and 
a number $\epsilon>0$, there is a holomorphic Legendrian 
immersion $F: M\to X$ satisfying the conditions in Theorem \ref{th:dense},
which can be chosen injective if the interpolation conditions allow it.

In particular, every bordered Riemann surface admits an injective holomorphic 
Legendrian immersion to $X$ whose image is everywhere dense in $X$.
\end{theorem}

To justify the last statement in the theorem, recall that every holomorphic contact
subbundle $\xi\subset TX$ is given in suitable local holomorphic coordinates 
at any point $p\in X$ as the kernel of the standard holomorphic contact form
$
	\alpha=dz+\sum_{i=1}^n x_i dy_i,
$
on $\C^{2n+1}$ with $2n+1=\dim X$, 
where $x=(x_1,\ldots,x_n)\in\C^n,\ y=(y_1,\ldots,y_n)\in\C^n$, and $z\in \C$ 
are complex coordinates on $\C^{2n+1}$.
(This is the holomorphic Darboux theorem; see 
Alarc\'on et al. \cite[Appendix]{AlarconForstnericLopez2017CM}).
It was shown in \cite{AlarconForstnericLopez2017CM} that every compact 
bordered Riemann surface $\overline M$ admits a holomorphic 
Legendrian embedding $g:\overline M\hra (\C^{2n+1},\alpha)$.
Composing $g$ by the contact isomorphism 
$(x,y,z)\mapsto (tx,ty,t^2z)$ on $(\C^{2n+1},\alpha)$ for a suitable $t>0$ ensures 
that $g(\overline M)$ lies in the image of the local chart, 
so we obtain a holomorphic Legendrian embedding $f:\overline M\hra (X,\xi)$. 
Applying the first part of the theorem to $f$ yields the result.

%
%
% 	ACKNOWLEDGEMENTS
%
%
\medskip
\noindent {\bf Acknowledgements.} 
Alarc\'on is partially supported by the State Research Agency (AEI) via the grants 
PID2020-117868GB-I00 and PID2023-150727NB-I00, and the ``Maria de Maeztu'' Unit of Excellence IMAG, reference CEX2020-001105-M, funded by MICIU/AEI/10.13039/501100011033/, 
and the Junta de Andaluc\'ia grant P18-FR-4049, Spain.
Forstneri\v c is supported by the European Union 
(ERC Advanced grant HPDR, 101053085) and grants P1-0291 and N1-0237 from 
ARIS, Republic of Slovenia. 
The authors thank Rafael Andrist, Kyle Broder, 
Fr\'ed\'eric Campana, J\'anos Koll\'ar, 
Uro\v s Kuzman, Finnur L\'arusson, Thomas Peternell, 
and Alexandre Sukhov for consultations and advice.
They also thank Olivier Benoist and Olivier Wittenberg for the communication 
regarding their work \cite{BenoistWittenberg2025}, and 
Song-Yan Xie for the communication regarding the papers
\cite{GuoXie2024MA,Xie2024IMRN}.

%
% 	ACKNOWLEDGEMENTS
%
%
\medskip
\noindent {\bf Statements and Declarations.}
On behalf of all authors, the corresponding author 
states that there is no conflict of interest. 
The manuscript has no associated data.

%%%%%%%%%%
%%%%%%%%%%
%%%%%%%%%%
%%%%%%%%%%   THE BIBLIOGRAPHY
%%%%%%%%%%
%%%%%%%%%%

%{\bibliographystyle{abbrv} \bibliography{references}} 
%\begin{comment}

%\end{comment}

\end{document}